\documentclass{amsart}[12pt]

\usepackage{amssymb,xcolor}
\usepackage{amsmath}
\usepackage{amsfonts}
\usepackage{paralist}

\usepackage[pdftex]{hyperref}
\hypersetup{
    colorlinks=true, %set true if you want colored links
    linktoc=all,     %set to all if you want both sections and subsections linked
    linkcolor=blue,  %choose some color if you want links to stand out
}

\theoremstyle{plain}
\newtheorem{theorem}{Theorem}[section]
\newtheorem{Claim}[theorem]{Claim}

\newtheorem{Fact}[theorem]{Fact}
\newtheorem{lemma}[theorem]{Lemma}

\newtheorem{proposition}[theorem]{Proposition}
\newtheorem{fact}[theorem]{Fact}

\newtheorem{question*}{Question}
\theoremstyle{definition}
\newtheorem{definition}[theorem]{Definition}
\newtheorem{remark}[theorem]{Remark}
\newtheorem{observation}[theorem]{Observation}

\newcommand{\trcl}[1]{\ensuremath{\mathrm{trcl}(\{#1\})}}

\newcommand{\dom}[1]{\ensuremath{\mathrm{dom}}(#1)}

\newcommand{\set}[2]{\ensuremath{\{#1 \,|\, #2 \}}}
\newcommand{\seq}[2]{\ensuremath{\langle #1 \,|\, #2 \rangle}}

\newcommand{\sub}{\subseteq}

\newcommand{\la}{\langle}
\newcommand{\ra}{\rangle}

\newcommand{\then}{\rightarrow}

\newcommand{\bb}{\mathbb}

\newcommand{\beq}{\begin{equation}}
\newcommand{\eeq}{\end{equation}}
\newcommand{\brm}{\begin{remark}\begin{rm}}
\newcommand{\erm}{\end{rm}\end{remark}}

\newcommand{\bce}{\begin{compactenum}}
\newcommand{\ece}{\end{compactenum}}

\newcommand{\Ord}{\mathrm{Ord}}

\newcommand{\Add}{\mathrm{Add}}

\newcommand{\R}{\bb{R}}
\newcommand{\Q}{\bb{Q}}
\renewcommand{\P}{\bb{P}}
\newcommand{\TP}{{\sf TP}}
\newcommand{\wTP}{\mathsf{wTP}}
\newcommand{\SH}{\sf{SH}}

\newcommand{\U}{\bb{U}}
\newcommand{\T}{\bb{T}}
\newcommand{\M}{\bb{M}}
\newcommand{\x}{\times}

\newcommand{\Col}{\mathrm{Col}}
\renewcommand{\S}{\mathbb{S}}

\newcommand{\CU}{\mathrm{CU}}

\newcommand{\ZFC}{\sf ZFC}

\newcommand{\GCH}{\sf GCH}

\newcommand{\SR}{{\sf SR}}
\newcommand{\AP}{{\sf AP}}

\newcommand{\CSR}{\mathsf{CSR}}

\newcommand{\cof}{\mathrm{cof}}

\newcommand{\cf}{\mathrm{cf}}

\newcommand{\lam}{\lambda}

\newcommand{\N}{\bb{N}}

\newcommand{\rest}[0]{\!\restriction\!}

\title[Trees and Stationary Reflection at Double Succ.'s of Regular Card.'s]{Trees and stationary reflection at double successors of regular cardinals}

\author{Thomas Gilton}
\address{University of Pittsburgh\\Department of Mathematics\\166 Thackeray Avenue\\Pittsburgh, PA 15213\\United States of America}
\email{tdg25@pitt.edu}

\author{Maxwell Levine}
\address{University of Freiburg\\Institute of Mathematics\\Ernst-Zermelo-Strasse 1\\Freiburg im Breisgau 79104\\Germany}
\email{maxwell.levine@mathematik.uni-freiburg.de}

\author{{\v S}{\'a}rka Stejskalov{\'a}}
\address{Charles University, Department of Logic, Celetn{\' a} 20, Prague 1, 116 42, Czech Republic}
\address{Institute of Mathematics, Czech Academy of Sciences, {\v Z}itn{\'a} 25, Prague 1, 115 67, Czech Republic}
\email{sarka.stejskalova@ff.cuni.cz}

\thanks{The first author proved results that were independently obtained by the second and third author using slightly different methods. In particular, Theorem \ref{tp-csr-ap} was obtained independently by the second and third authors on the one hand and by the first author (in joint work with Omer Ben-Neria) on the other. When we found out about this, we decided to work together to shape our work into a cohesive statement, combining material as necessary. Theorems \ref{tp-csr-notap} and \ref{csr-notsh} are in the first author's thesis.\\ \indent The second author was supported by FWF Grant \#28157. The third author was partially supported by FWF/GA{\v C}R grant \emph{Compactness principles and combinatorics} (19-29633L)}

\begin{document}

\begin{abstract} We obtain an array of consistency results concerning trees and stationary reflection at double successors of regular cardinals $\kappa$, updating some classical constructions in the process. This includes models of $\CSR(\kappa^{++})\wedge \TP(\kappa^{++})$ (both with and without $\AP(\kappa^{++})$) and models of the conjunctions $\SR(\kappa^{++}) \wedge \wTP(\kappa^{++}) \wedge \AP(\kappa^{++})$ and $\neg \AP(\kappa^{++}) \wedge \SR(\kappa^{++})$ (the latter was originally obtained in joint work by Krueger and the first author \cite{GilKru:8fold}, and is here given using different methods). Analogs of these results with the failure of $\SH(\kappa^{++})$ are given as well. Finally, we obtain all of our results with an arbitrarily large $2^\kappa$, applying recent joint work by Honzik and the third author.\end{abstract}

\maketitle

\section{Introduction}

The study of compactness properties of successor cardinals is a prominent theme in set theory. Much of this research concerns the consistency strength required to obtain certain combinations of these properties. In this paper, we will consider variants of the tree property together with variants of the stationary reflection property at double successors of regular cardinals, and we will construct models using optimal large cardinal hypotheses. Our results will all use some variant of Mitchell forcing. Throughout the paper, we will divide these results into those which require weakly compact  cardinals and those which require Mahlo cardinals.

Mitchell and Silver proved that if $\lambda$ is a weakly compact cardinal, then for every regular $\kappa<\lambda$, there is a forcing extension in which $\lambda = \kappa^{++}$ and $\TP(\kappa^{++})$ holds \cite{M:tree}. (Exact definitions of these combinatorial principles are given below in Section~\ref{sec-definitions}.) It is known that $\SR(\kappa^{++})$ holds in Mitchell's model of the tree property \cite{8fold}, although $\CSR(\kappa^{++})$ fails.\footnote{If $\lambda,\kappa$ are regular with $\lambda > \kappa^{++}$, we can find a stationary subset $S$ of $\lambda\cap \cof(\kappa)$ and a stationary subset $A$ of $\lambda\cap\cof(\kappa^+)$ such that $S$ does not reflect at any point of $A$ (see Jech, Exercise 23.12 \cite{JECHbook}). If $\lambda$ is weakly compact, then $\M(\kappa,\lambda)$ is $\lambda$-cc, so $S$ and $A$ are still stationary in the extension. Since $\kappa$ and $\kappa^+$ are preserved by Mitchell forcing, we see that $S\sub \kappa^{++}\cap \cof(\kappa)$ and  $A\sub\kappa^{++}\cap\cof(\kappa^+)$ in $V[\M(\kappa,\lambda)]$, moreover $S$ cannot reflect at any point of $A$.} On the other hand, Magidor proved in \cite{M:sr} that assuming the existence of a weakly compact $\lambda$ and a regular $\kappa<\lambda$, there is a forcing extension in which $\CSR(\kappa^{++})$ holds. In Magidor's model, $\wTP(\kappa^{++})$ fails because $2^{\kappa} = \kappa^{+}$. Mitchell and Silver as well as Magidor showed that their respective results require a weakly compact cardinal.

This raises the natural question of whether $\TP(\kappa^{++})$ and $\CSR(\kappa^{++})$ can be obtained simultaneously, and moreover whether a weakly compact cardinal is sufficient for this purpose. We prove in this paper that one can indeed obtain both properties from a weakly compact cardinal:

\begin{theorem}\label{tp-csr-ap}
Assume that $\kappa$ is regular with $\kappa^{<\kappa} = \kappa$ and $\lambda>\kappa$ is a weakly compact cardinal. Then there is a model where $\lambda = \kappa^{++}$ and both $\TP(\kappa^{++})$ and $\CSR(\kappa^{++})$ hold, and $\AP(\kappa^{++})$ holds as well.
\end{theorem}

In order to prove Theorem~\ref{tp-csr-ap}, we force with Mitchell forcing followed by an iteration of club-adding posets. This requires us to show that the Mitchell poset has an absorption property similar to but distinct from that of the L{\'e}vy collapse. The variant of Mitchell forcing used for proving Theorem~\ref{tp-csr-ap} in fact forces $\AP(\kappa^{++})$, and although there are alternate versions that force $\AP(\kappa^{++})$ to fail \cite{8fold}, it is unclear whether they have an absorption property. Therefore we employ a guessing variant of Mitchell forcing---similar to the one used by Cummings and Foreman \cite{CUMFOR:tp} and used by the first author in his thesis \cite{Gilton-thesis}---to obtain a similar model in which $\AP(\kappa^{++})$ fails. Variations of Mitchell forcing with guessing were introduced by Abraham \cite{ABR:tree}.

\begin{theorem}\label{tp-csr-notap} Assume that $\kappa$ is regular with $\kappa^{<\kappa}=\kappa$ and $\lambda>\kappa$ is a weakly compact cardinal. Then there is a model where $\lambda=\kappa^{++}$ and where $\CSR(\kappa^{++}) \wedge \TP(\kappa^{++})\wedge \neg\AP(\kappa^{++})$ holds.\end{theorem}

We next obtain a similar result with the failure of the Suslin hypothesis---recall that the Suslin hypothesis \emph{fails} if a Suslin tree does in fact exist. Since the existence of Suslin trees is a typical consequence of square principles, which also imply the failure of various forms of stationary reflection, one can view the failure of the Suslin hypothesis as being naturally in tension with stationary reflection properties. Nonetheless, we find that there is compatibility here.

\begin{theorem}\label{csr-notsh} Assume that $\kappa$ is regular with $\kappa^{<\kappa}=\kappa$ and $\lambda>\kappa$ is a weakly compact cardinal. Then there is a model in which $\lambda=\kappa^{++}$ and $\CSR(\kappa^{++}) \wedge \neg \SH(\kappa^{++}) \wedge \neg \AP(\kappa^{++})$ holds. (In particular, $\TP(\kappa^{++})$ fails.)\end{theorem}

Upon reading through our construction, the reader may observe that one could use the standard L{\'e}vy collapse in the construction to obtain the conjunction of $\neg \SH(\kappa^{++})$ with $\CSR(\kappa^{++})$, but it is the failure of $\AP(\kappa^{++})$ for which our use of Mitchell forcing is needed.

We proceed to consider compactness properties that can be obtained from Mahlo cardinals. Harrington and Shelah showed that $\SR(\kappa^{++})$ is equiconsistent with a Mahlo cardinal \cite{HS:sr}, and Mitchell and Silver showed that $\wTP(\kappa^{++})$ is also equiconsistent with a Mahlo cardinal \cite{M:tree}. We simultaneously achieve both properties from a Mahlo cardinal, together with the approachability property, using techniques similar to those used to prove Theorem~\ref{tp-csr-ap}.

\begin{theorem}\label{wtp-sr-ap}
Assume that $\kappa$ is regular with $\kappa^{<\kappa} = \kappa$ and $\lambda>\kappa$ is a Mahlo cardinal. Then there is a model in which $\lambda = \kappa^{++}$ and in which $\wTP(\kappa^{++})$, $\SR(\kappa^{++})$, and $\AP(\kappa^{++})$ all hold.\end{theorem}

It is well-known that the failure of $\AP(\mu^+)$ implies $\wTP(\mu^+)$ for an arbitrary cardinal $\mu$ using the equivalence of $\wTP(\mu^+)$ with the failure of weak square $\square_\mu^*$ and the fact that $\square_\mu^*$ implies $\AP(\mu^+)$. Krueger and the first author constructed a model with $\SR(\aleph_2)$ and the failure of $\AP(\aleph_2)$ from a Mahlo cardinal \cite{GilKru:8fold}, so in their model $\wTP(\aleph_2)$ holds. More specifically, they use a \emph{disjoint stationary sequence on $\aleph_2$} (for the definition see \cite{GilKru:8fold}) because this notion is amenable to preservation theorems. However, Theorem~\ref{wtp-sr-ap} separates $\wTP(\kappa^{++})$ and $\neg\AP(\kappa^{++})$ in the presence of $\SR(\kappa^{++})$.

Using the guessing variant of Mitchell forcing employed for proving Theorem~\ref{tp-csr-notap}, together with an application of the principle $\diamondsuit_\lambda(\text{Reg})$, we show that we can modify the proof of Theorem~\ref{wtp-sr-ap} to obtain the result of Krueger and the first author without using disjoint stationarity.

\begin{theorem}\label{gk}
Assume that $\kappa$ is regular with $\kappa^{<\kappa} = \kappa$ and $\lambda>\kappa$ is a Mahlo cardinal. Then there is a model in which $\lambda = \kappa^{++}$ and in which $\SR(\kappa^{++})$ holds and $\AP(\kappa^{++})$ fails.\end{theorem}

These techniques can also be used to obtain a Mahlo analog of Theorem~\ref{csr-notsh}:

\begin{theorem}\label{sr-notsh} Assume that $\kappa$ is regular with $\kappa^{<\kappa} = \kappa$ and $\lambda>\kappa$ is a Mahlo cardinal. Then there is a model in which $\lambda=\kappa^{++}$ and $\SR(\kappa^{++}) \wedge \neg \SH(\kappa^{++}) \wedge \neg \AP(\kappa^{++})$ holds.\end{theorem}

Our above-mentioned models all have the property that $2^{\kappa} = \kappa^{++}$. We end the paper by showing that we can in fact obtain both results with large $2^{\kappa}$.

\begin{theorem}\label{wc-large-cont} There are models for the conclusions of Theorems \ref{tp-csr-ap}, \ref{tp-csr-notap}, and \ref{csr-notsh} in which $2^{\kappa}=\mu>\lambda$, for arbitrary $\mu$ of cofinality greater than $\kappa$.\end{theorem}

\begin{theorem}\label{mahlo-large-cont}
There are models for the conclusions of Theorems \ref{wtp-sr-ap}, \ref{gk}, and \ref{sr-notsh} in which $2^{\kappa}=\mu>\lambda$ for arbitrary $\mu$ of cofinality greater than $\kappa$.
\end{theorem}

In their paper, Krueger and the first author obtain their result with a large continuum by using an unpublished preservation theorem of Neeman for showing that the Cohen forcing at $\aleph_0$ does not destroy $\SR(\aleph_2)$ \cite{GilKru:8fold}. However, Neeman's proof does not seem to generalize to Cohen forcing at an arbitrary regular $\kappa>\omega$. We overcame this obstacle by using a preservation theorem by Honzik and the third author (see \cite{HS:u}). This preservation theorem comes with an analog used to prove Theorem~\ref{wc-large-cont}.

This paper will be organized as follows: Section~\ref{sec-background} will cover background; Section~\ref{sec-techniques} will deal with our modest attempts to improve and bolster existing techniques for dealing with forcing quotients; in Section~\ref{sec-wc-results} we will prove the results that require weakly compact cardinals; in Section~\ref{sec-mahlo-results} we will prove the results that require Mahlo cardinals; and finally in Section~\ref{sec-large-cont} we will show how to obtain our results with large $2^\kappa$. Some of the presented arguments will become routine over the course of the paper. We will present earlier versions in detail and later versions with reference to the earlier ones.

\section{Background}\label{sec-background}

We assume that the reader is familiar with the basics of forcing \cite{JECHbook}.

\subsection{Basic Definitions}\label{sec-definitions}

Here we explicitly define our basic compactness concepts, and clarify the notation we use for them.

\begin{definition}\

\begin{enumerate}

\item We say that a regular cardinal $\lambda$ has the \emph{tree property}, denoted $\TP(\lambda)$, if every tree of height $\lambda$ and with levels of size $<\lambda$ has a cofinal branch.

\item Let $\mu$ be an infinite cardinal and let $T$ be a tree of height $\mu^+$ with levels of size $\le \mu$. We say that $T$ is a \emph{special} $\mu^+$-\emph{Aronszajn tree} if there is a function $f:T \to \mu$ such that if $x$ and $y$ are elements of the tree that are related in the ordering of the tree, then $f(x) \ne f(y)$.

\item The \emph{weak tree property}, denoted $\wTP(\mu^+)$, holds at $\mu^+$ for some infinite cardinal $\mu$ if there are no special $\mu^+$-Aronszajn trees.

\end{enumerate}\end{definition}

\begin{definition} Assume that $\mu$ is regular.\footnote{Alternate forms of these definitions can be applied to inaccessibles or successors of singulars.}

\begin{enumerate}

\item The \emph{stationary reflection property} at $\mu^+$ is denoted $\SR(\mu^+)$, and it holds if for every stationary subset $S \subseteq \mu^+\cap\cof(<\mu)$, there is some $\alpha<\mu^+$ (called a \emph{reflection point}) of cofinality $\mu$ such that $S \cap \alpha$ is stationary in $\alpha$.

\item The \emph{club stationary reflection property} at $\mu^+$ is denoted $\CSR(\mu^+)$, and it holds if for every stationary $S\sub\mu^+\cap\cof(<\mu)$, there is a \emph{club} $C \subset \mu^+$ such that every point $\alpha \in C$ of cofinality $\mu$ is a reflection point of $S$.

\end{enumerate}\end{definition}

\begin{definition} Let $\lambda$ be regular and let $\vec{a}=\seq{a_\alpha}{\alpha<\lambda}$ be a sequence of bounded subsets of $\lambda$. An ordinal $\gamma<\lambda$ is \emph{approachable} with respect to $\vec{a}$ if there is an unbounded subset $A\sub\gamma$ of order type $\cf( \gamma)$ such that for all $\beta<\gamma$ there is $\alpha<\gamma$ with $A\cap\beta=a_\alpha$. The \emph{approachability property} holds at $\lambda$ if there is a club $C \subseteq \lambda$ such that every $\gamma\in C$ is approachable with respect to $\vec{a}$, and we denote this $\AP(\lambda)$.\end{definition}

\begin{fact} Suppose $W$ is a model of $\ZFC-\textup{\textsf{Powerset}}$ and $\AP(\lambda)$ holds for a regular $\lambda$. If $W' \supseteq W$ is an extension such that $W' \models$``$\lambda$ is regular'', then $W' \models \AP(\lambda)$.\end{fact}

\begin{definition} If $\lambda$ is regular, a $\lambda$-\emph{Suslin tree} is a tree $T$ with height $\lambda$ and levels of size $<\lambda$, so that for all antichains $A \subset T$, we have $|A|<\lambda$. (This is equivalent to saying that the poset $(T,<_T)$ is $\lambda$-cc.) If $\lambda$ does not carry a Suslin tree, we say that the \emph{Suslin hypothesis} holds for $\lam$, and we denote this $\SH(\lambda)$.\end{definition}

Every normal $\lambda$-Suslin tree is a $\lambda$-Aronszajn tree, so $\TP(\lambda)$ implies $\SH(\lambda)$.

\subsection{Mitchell Forcing}\label{Mitchell-exposition} All of our results make use of some variant of Mitchell forcing. The classical Mitchell forcing takes the following form:

\begin{definition}\label{Def:Mitchell}
Let $\kappa$ be a regular cardinal and $\lambda$ an inaccessible cardinal with $\lambda > \kappa$. The \emph{Mitchell forcing} at $\kappa$ of length $\lambda$, denoted by $\M(\kappa,\lambda)$, is the set of all pairs $(p,q)$ such that:

\begin{enumerate}
\item $p$ is in  $\Add(\kappa,\lambda)$;
\item $q$ is a function with $\dom{q}\sub\lambda$ of size at most $\kappa$;
\item for every $\beta\in\dom{q}$, $1_{\Add(\kappa,\beta)}\Vdash q(\beta)\in\dot{\Add}(\kappa^+,1)$
where $\dot{\Add}(\kappa^+,1)$ is the canonical $\Add(\kappa,\beta)$-name for adding a Cohen subset of $\kappa^+$.
\end{enumerate}

A condition $(p,q)$ is stronger than $(p',q')$ if:
\begin{enumerate}[(i)]
\item $p\leq_{\Add(\kappa,\lambda)} p'$;
\item $\dom{q}\supseteq \dom{q'}$ and for every $\beta\in\dom{q'}$, ${p}\rest{\beta}\Vdash q(\beta)\leq q'(\beta)$.
\end{enumerate}
\end{definition}

It can be shown that approachability holds in extensions by the classical Mitchell forcing:

\begin{Fact} \cite{8fold} $V[\M(\kappa,\lambda)] \models \AP(\lambda)$ .\end{Fact}

Assuming that $\kappa < \lambda$, $\lambda$ is inaccessible, and $\kappa$ is regular, $\M(\kappa,\lambda)$ is $\lambda$-Knaster (hence it preserves $\lambda$) and $\kappa$-closed; $\M(\kappa,\lambda)$ also collapses cardinals in the open interval $(\kappa^+,\lambda)$. Furthermore, if $\kappa^{<\kappa}=\kappa$, $\M(\kappa,\lambda)$ also forces $2^\kappa=\lambda = \kappa^{++}$. A helpful way to see this, which we will use throughout the paper, is due to an analysis of Abraham. Namely, there is a projection $\pi:\Add(\kappa,\lambda)\x \U \then \M(\kappa,\lambda)$ where $\U$ is a $\kappa^+$-closed forcing \cite{ABR:tree}. Note that since $\kappa^{<\kappa}=\kappa$, $\Add(\kappa,\lambda)$ is also square-$\kappa^+$-cc, meaning that $\Add(\kappa,\lambda) \times \Add(\kappa,\lambda)$ has the $\kappa^+$-chain condition. This projection analysis is extremely useful. In particular we find that any downwards-absolute notion that is preserved by a product of a square-$\kappa^+$-cc forcing with a $\kappa^+$-closed forcing is also preserved by $\M(\kappa,\lambda)$. In particular, one obtains the preservation of $\kappa^+$ by Easton's Lemma:

\begin{fact}\label{L:Easton}\emph{(Easton)}
Let $\kappa>\aleph_0$ be a regular cardinal, and assume that $\P$ and $\Q$ are forcing notions, where $\P$ is $\kappa$-cc and $\Q$ is $\kappa$-closed. Then the following hold:
\begin{enumerate}
\item\label{L:Easton-cc} $\Q$ forces that $\P$ is $\kappa$-cc.
\item $\P$ forces that $\Q$ is $\kappa$-distributive.
\end{enumerate}
\end{fact}

Furthermore, $\M(\kappa,\lambda)$ preserves stationary subsets of $\mu \cap \cof(\le\!\kappa)$ for regular $\mu \ge \kappa^+$ by the following two facts and Fact~\ref{L:Easton}(\ref{L:Easton-cc}):

\begin{fact}\label{F:St_cc}
Let $\lambda$ be regular. If $\P$ is $\lambda$-cc, then it preserves stationary subsets of $\lambda$.
\end{fact}

\begin{fact}\label{F:St_Closed_cof}
Let $\lambda = \kappa^{++}$ and let $\P$ be $\kappa^+$-closed. If $S\sub \lambda \cap\cof(\le\kappa)$ is stationary, then $\P$ preserves the stationarity of $S$.\footnote{This item should be read in light of the fact that $\kappa^{++}\cap\cof(\leq\kappa)\in I[\kappa^{++}]$ by \cite{CUM:comb}.}\end{fact}

There are natural projections from the Mitchell forcing of length $\lambda$ to Mitchell forcings of shorter lengths. If $\delta < \lambda$ is inaccessible and $\bar G$ is $\M(\kappa,\delta)$-generic, we write $\M(\kappa,\lambda)/\bar G$ to represent the natural quotient. These projections also commute with Abraham's projection analysis. More precisely, in $V[\bar{G}]$, $\M(\kappa,\lambda)/\bar G$ is (a dense subset of) a projection of the product of a square-$\kappa^+$-cc forcing with a $\kappa^+$-closed forcing. Because $V[\bar G] \models \text{``} 2^\kappa = \delta$'', it follows that $\M(\kappa,\lambda)/\bar G$ does not add cofinal branches to $\mu$-trees over $V[\bar G]$ for $\kappa^+ \le \mu \le \delta$ by the following two facts:

\begin{fact}[Silver]\label{silverslemma}
Let $\kappa$ be regular, let $\delta$ be an ordinal of cofinality at least $\kappa^+$, and let $\P$ be $\kappa^{+}$-closed. If $T$ is a tree of height $\delta$ and levels of size $<2^\kappa$, then forcing with $\P$ does not add cofinal branches to $T$.\end{fact}

\begin{fact}\label{F:Square_cc}\cite{Devlin-aleph1trees, Unger:II}
Let $\lambda$ be regular, and let $\P$ be square-$\lambda$-cc. If $T$ is a tree of height $\lambda$, then forcing with $\P$ does not add cofinal branches to $T$.
\end{fact}

If we combine these two facts with Easton's Lemma, we find that forcing with a product of a square-$\kappa^+$-cc forcing with a $\kappa^+$-closed forcing does not add cofinal branches to a tree of height of cofinality at least $\kappa^+$ with levels of size $<2^\kappa$.

Hence we have the following preservation properties for these quotients:

\begin{Fact}\label{Mitchell-preservation-summary} Suppose $\kappa^{<\kappa} = \kappa < \lambda$ and that $\P=  \M(\kappa,\lambda)/\bar{G}$ where $\bar{G}$ is $\M(\kappa,\delta)$-generic for inaccessible $\delta<\lambda$.

\begin{enumerate}

\item $\P$ preserves $\kappa^+$.

\item $\P$ does not add cofinal branches to $\mu$-trees if $\kappa^+ \le \mu \le \delta$.

\item $\P$ preserves stationary subsets of $\mu \cap \cof(\le\!\kappa)$ if $\kappa^+ \le \mu\le\delta$. 

\end{enumerate}\end{Fact}

\subsection{Large Cardinals and Lifting Embeddings}\label{largecardsandembs}

As mentioned in the introduction, large cardinals are necessary to obtain our results.

\begin{definition} An uncountable cardinal $\lambda$ is \emph{Mahlo} if it is inaccessible and the set of all regular cardinals below $\lambda$ is stationary in $\lambda$.
\end{definition}

\begin{definition} If $j:M \to N$ is an elementary embedding between transitive sets, we say that $\delta$ is its \emph{critical point} if $\delta$ is the least ordinal such that $j(\delta)>\delta$.\end{definition}

\begin{definition}\cite{CUMhandbook} An uncountable cardinal $\lambda$ is \emph{weakly compact} if for every transitive set $M$ of size $\lambda$ such that $\lambda\in M$ and ${^{<\lambda}} M\sub M$, there is an elementary embedding $j:M\then N$ with critical point $\lambda$, where $N$ is transitive of size $\lambda$ and $^{<\lambda}N\sub N$.
\end{definition}

The technical aspects of our results involve the extension of embeddings:

\begin{fact}\label{weakliftinglemma} Suppose $j: M \to N$ is an elementary embedding between transitive models of (enough of) $\ZFC$ and $\P \in M$. Let $G$ be $\P$-generic over $M$ and let $H$ be $j(\P)$-generic over $N$ such that $j[G] \subseteq H$. Then $j$ can be extended to $j^+:M[G] \to N[H]$ such that $j^+(G)=H$ and $j^+\rest M=j$.\end{fact}

The notation $j^+$ is usually suppressed, and lifts are denoted $j$. We will also need a stronger version of this fact using master conditions:

\begin{definition} Suppose $j: M \to N$ is an elementary embedding between transitive models of (enough of) $\ZFC$ and $\P \in M$. We say that $q \in j(\P)$ is a \emph{strong master condition} for $j$ and $\P$ if for every dense $D \subseteq \P$ with $D\in M$ there is some $p \in D$ such that $q \le j(p)$. In particular, $q$ is a strong master condition if $G$ is $\P$-generic over $M$ and $q \le j(p)$ for all $p \in G$, i.e. if $q$ is a \emph{lower bound} for $j[G]$. \end{definition}

Observe that if $q$ is a strong master condition for $j$ and $\P$ in this sense, then $\{ p \in \P : q \le j(p) \}$ is $\P$-generic over $M$.

\begin{fact}\label{strongliftinglemma} \cite{CUMhandbook} Let $j: M \to N$ be an elementary embedding between transitive models of (enough of) $\ZFC$ with critical point $\kappa$, and let $\P \in M$. Suppose that $G$ is $\P$-generic over $M$ and that there is some $q \in j(\P)$ such that $q$ is a lower bound for $j[G]$. Then $M$ and $M[G]$ both have the same $<\kappa$-sequences of ordinals.\end{fact}

\subsection{Club-Adding Iterations}

In this section, we define a standard tool for obtaining stationary reflection properties and state some useful facts.

\begin{definition} If $\lambda$ is a regular cardinal and $X \subset \lambda$, we define the \emph{club-adding poset} $\CU(X)$ as the poset of closed, bounded subsets $c \subset X$. If $c,d \in \CU(X)$, then $c \le d$ if and only if $c$ end-extends $d$, i.e. $d = c \cap (\max d + 1)$.\end{definition}

We define iterations $\seq{\P_\alpha,\dot \Q_\alpha}{\alpha < {\lambda^+}}$ as in the Handbook of Set Theory \cite{CUMhandbook} so that conditions have support of size $<\lambda$. For convenience, we will use:

\begin{definition}\label{def:SCAI} We say that $\seq{\P_\alpha,\dot \Q_\alpha}{\alpha < {\lambda^+}}$ is a \emph{standard club-adding iteration of length} $\lambda^+$ if $\P_\alpha \Vdash \text{``}\dot \Q_\alpha = \CU(\dot X), \dot X \subset \lambda$'' for all $\alpha < \lambda^+$ and the iteration has $<\lambda$-support.\end{definition}

\begin{fact}\label{iterationchaincondition} Suppose that $\lambda$ is regular, $\lambda^{<\lambda} = \lambda$, and that $\P := \seq{\P_\alpha,\dot \Q_\alpha}{\alpha < \lambda^+}$ is a standard club-adding iteration of length $\lambda^+$. Then $\P$ has the $\lambda^+$-chain condition.\end{fact}

One of the useful properties of standard club-adding iterations is that $\lambda$-sized sets show up at initial stages. We will eventually use this fact without comment:

\begin{fact}\label{nicenamestrick} If $\P_{\lambda^+}$ is a standard club-adding iteration of length $\lambda^+$ and $f: \lambda \to \Ord$ is defined in $V[\P_{\lambda^+}]$, then there is some $\alpha<\lambda^+$ such that $f \in V[\P_\alpha]$.\end{fact}

\section{Some Additional Techniques}\label{sec-techniques}

Our results follow classical constructions of Magidor and of Harrington and Shelah. The technical challenges involved in these sorts of constructions involve making sure that the quotient forcings used for lifting certain embeddings are well-behaved. Here, we outline some additional techniques that will be applied in the remaining sections.

\subsection{Absorption for Classical Mitchell Forcing}\label{absorptionsection}

We introduce a variant of the Absorption Lemma, which originally appeared in Solovay's construction of a model in which all sets of reals are Lebesgue-measurable \cite{S:m}. It shows that certain sets of ordinals can be ``absorbed'' into an extension by a Levy Collapse. It is also used in the constructions of Harrington and Shelah and of Magidor mentioned above. For our purposes, we need an Absorption Lemma that works with the classical variant of Mitchell Forcing.

\begin{definition}\cite{CUMhandbook} If $G$ is $\P$-generic over $V$, we say $G$ \emph{induces} a generic for $\Q$ if $V[G]$ contains a filter $H$ that is $\Q$-generic over $V$. Two posets $\P$ and $\Q$ are \emph{forcing-equivalent} if a generic for $\P$ induces a generic for $\Q$ and vice versa. We denote this $\P \simeq \Q$.\footnote{With straightforward modifications, we can use forcing equivalence defined by means of the isomorphism of the Boolean completions of the partial orders.}\end{definition}

\begin{Fact}\label{F:equivalent} \cite{CUMhandbook}
Let $\mu$ be a regular cardinal and assume $2^\mu=\mu^+$. Then all $\mu^+$-closed, nontrivial, separative forcings of cardinality $\mu^+$ are forcing-equivalent. Specifically, any such poset is forcing-equivalent to $\Add(\mu^+,1)$.\end{Fact}

\begin{lemma}\label{absorption}(Absorption for Mitchell Forcing) Suppose $V$ is any transitive model of set theory, that $\bar G$ is $\M(\kappa,\delta)$-generic for an inaccessible cardinal $\delta<\lambda$, and that $\Q$ is a $\kappa^+$-closed, separative poset of cardinality $\delta$ in $V[\bar G]$. Then $\M(\kappa,\lambda)/\bar G \simeq \Q \times \M(\kappa,\lambda)/\bar G$ in $V[\bar G]$.\end{lemma}

\begin{proof} First, we claim that in $V[\bar G]$, $\M/\bar G \simeq \Add(\kappa^+,1) \x \M^*$ where $\M^*=\set{(p,q)\in\M(\kappa,\lambda)/\bar{G}}{\delta\notin \dom{q}}$. The set $D(\M/\bar G) = \{(p,q) \in \M/\bar G:\delta \in \dom{q}\}$ is clearly dense in $\M/\bar G$. Define $i: D(\M/\bar G) \to \Add(\kappa^+,1)\x \M^*$ by $i :(p,q) \mapsto (q(\delta),(p,q\rest(\lambda\setminus\{\delta\})))$. It is routine to verify that $i$ is a dense embedding.

Next, we claim that in $V[\bar G]$, $\Q \times \Add(\kappa^+,1) \simeq \Add(\kappa^+,1)$. It is enough to show that if $h$ is $\Add(\kappa^+,1)$-generic over $V[\bar G]$, then in $V[\bar G \ast h]$ there is a filter $h'$ that is $\Q$-generic over $V[\bar G]$. Since $\Add(\kappa^+,1) \simeq \Add(\kappa^+,1) \times \Add(\kappa^+,1)$ (by Fact~\ref{F:equivalent}, although this can be shown directly), there are $h_0,h_1$ such that $h_0 \times h_1$ is $\Add(\kappa^+,1) \times \Add(\kappa^+,1)$-generic over $V[\bar G]$ and such that $V[\bar G \ast h] = V[\bar G \ast (h_0 \times h_1)]$. The Mitchell forcing $\M(\kappa,\delta)$ collapses $\delta$ to be $\kappa^{++}$, i.e. $V[\bar G] \models \text{``}\delta=\kappa^{++}$'', and forcing with $\Add(\kappa^+,1)$ collapses $\delta$ further to have cardinality $\kappa^+$, so in $V[\bar G*h_0]$, $\Q$ has size $\kappa^+$ and $2^\kappa = \kappa^+$. Moreover, $\Q$ is still $\kappa^+$-closed because $\Add(\kappa^+,1)$ is $\kappa^+$-closed. By Fact~\ref{F:equivalent}, it follows that $\Q \simeq \Add(\kappa^+,1)$ in $V[\bar G \ast h_0]$. Hence there is a filter $h'$ that is $\Q$-generic over $V[\bar G \ast h_0]$ (and hence $V[\bar G]$) such that $V[\bar G \ast h_0][h_1] = V[\bar G \ast h_0][h']$. In particular, $h' \in V[\bar G \ast h]$.

It follows from both claims that in $V[\bar G]$ we have $\M(\kappa,\lambda)/\bar G \simeq\Add(\kappa^+,1) \x \M^* \simeq (\Q \times \Add(\kappa^+,1)) \x \M^*\simeq \Q \times (\Add(\kappa^+,1) \x \M^*)
 \simeq \Q \times \M(\kappa,\lambda)/ \bar G$. \end{proof}

We will often need the next lemma alongside the Absorption Lemma:

\begin{lemma}\label{preservation} Suppose $V$ is any transitive model of set theory, $\delta$ is inaccessible, $\bar G$ is $\M(\kappa,\delta)$-generic, $\Q$ is a $\kappa^+$-closed separative poset of cardinality $\delta$ in $V[\bar G]$ (which does not collapse $\delta$), and $H$ is $\Q$-generic over $V[\bar G]$. Then for $\lambda>\delta$ and regular $\mu$ such that $\kappa^+ \le \mu \le \delta$, forcing with $(\M(\kappa,\lambda)/\bar G)^{V[\bar G]}$ over $V[\bar{G}][H]$ preserves $(\kappa^+)^{V[\bar G][H]}$ as well as stationary subsets of $\mu \cap \cof(\le \kappa)$ from $V[\bar G][H]$, and does not add cofinal branches to $\mu$-trees over $V[\bar G][H]$.\end{lemma}

\begin{proof} Working in $V[\bar G]$, $\M(\kappa,\lambda)/\bar G$ is the projection of a product of a $\kappa^+$-closed poset and $\Add(\kappa,\lambda)$, a square-$\kappa^+$-cc poset. This is still the case in $V[\bar G][H]$ because of the $\kappa^+$-closure of $\Q$. Therefore, $\M(\kappa,\lambda)/\bar G$ has the usual preservation properties of Mitchell Forcing (Fact~\ref{Mitchell-preservation-summary}) over $V[\bar G][H]$.\end{proof}

\subsection{A Guessing Variant of Mitchell Forcing}

Using ideas that go back to Abraham \cite{ABR:tree}, we will also form variants of the Mitchell forcing which employ a guessing function of the form $\ell:\lambda \to V_\lambda$. The purpose will be to deal with quotients for constructions of models in which we intend for approachability to fail. The fact that the posets guessed by $\ell$ can be arranged to appear in certain quotients of this forcing serves the role of the absorption lemma in these cases. Details on this forcing can be found, among other places, in the first author's thesis \cite{Gilton-thesis}.

\begin{definition}\label{def:myMitchell} Let $A$ denote the set of inaccessible cardinals below $\lambda$, let $A^*$ denote $A\setminus \lim(A)$, and let $\ell:\lam\to V_\lam$ be given. We define the poset $\M_\ell(\kappa,\lambda) \rest\beta$ by induction on $\beta\in A$, setting $\M_\ell(\kappa,\lambda):= \bigcup_{\beta \in A^*} \M_\ell(\kappa,\lambda) \rest\beta$. Conditions in $\M_\ell(\kappa,\lambda)\rest\beta$ consist of triples $(a,f,g)$ where:\

\begin{enumerate}
\item $a\in\Add(\kappa,\beta)$;

\item $f$ is a partial function with $\dom f \subseteq A^*\cap\beta$ so that $|\dom f|\leq\kappa$;

\item for each $\alpha\in\dom f$, $f(\alpha)$ is an $\Add(\kappa,\alpha)$-name for a condition in $\Col(\kappa^+,\alpha)$;

\item $g$ is a partial function with $\dom g \subseteq A\cap\beta$ so that $|\dom g|\leq\kappa$;

\item for all $\alpha\in\dom g$, if $\ell(\alpha)$ is an $(\M_\ell\rest\alpha)$-name for a $\kappa^+$-closed poset then $g(\alpha)$ is an $(\M_\ell\rest\alpha)$-name for an element of $\ell(\alpha)$. (Otherwise $g(\alpha)$ is trivial.)

\end{enumerate}

We say that $(a',f',g')\leq (a,f,g)$ if and only if the following hold:

\begin{enumerate}[(i)]

\item $a\subseteq a'$;

\item $\dom f \subseteq\dom {f'}$ and $\dom g\subseteq\dom {g'}$;
\item for all $\alpha\in\dom f$, $a'\rest\alpha\Vdash_{\Add(\kappa,\alpha)}f'(\alpha)\leq_{\Col(\kappa^+,\alpha)}f(\alpha)$;
\item for all $\alpha\in\dom g $, $(a',f',g')\rest\alpha\Vdash_{\M_\ell \rest \alpha}g'(\alpha)\leq_{\ell(\alpha)}g(\alpha)$.

\end{enumerate}\end{definition}

It will be shown that $\M_\ell(\kappa,\lambda)$ forces the failure of $\AP(\lambda)$. In fact, the failure of $\AP(\lambda)$ will persist even when we follow $\M_\ell(\kappa,\lambda)$ with certain iterations. We will make arrangements so that under certain embeddings $j:M \to N$, the right stage of the iteration appears in the quotient as desired. Then we will obtain the failure of $\AP(\lambda)$ because of the fact that, since $\lambda$ is an inaccessible limit of inaccessibles (i.e., in $A\setminus A^*$), the forcing $j(\M_\ell(\kappa,\lambda))$ does not collapse $\lambda$ until stage $\lambda^*:=\min(j(A)\backslash(\lambda+1))$, after the additional subsets of $\kappa$ from $\Add(\kappa,\lambda^*)$ are added.

The conclusions of Fact \ref{Mitchell-preservation-summary} also hold for these guessing variants. We will frequently encounter the following situation: suppose $\delta\in(\kappa,\lambda)$ is inaccessible and $\delta^*$ is the least inaccessible greater than $\delta$. Suppose in addition that $\ell(\delta)=\P$ for some $\kappa^+$-closed forcing (otherwise the guessing variant is simply adding Cohen subsets of $\kappa$). We then obtain a quotient of the form:

\begin{equation*} \M_\ell(\kappa,\lambda) \simeq \M_\ell(\kappa,\delta) \ast({\P} \times\Add(\kappa,\delta^*))\ast  \R.
\end{equation*}

If $G$ is $\M_\ell(\kappa,\delta)$-generic over $N$ and $H \times I$ is $\P \times \Add(\kappa,\delta^*)$-generic over $N[G]$, then we write the quotient $\R$ as $\N_{\delta^*}$, following the notation of \cite{8fold}. In this case, we will make important use of the following:

\begin{fact}\cite{8fold} Tails of the form $\N_{\delta^*}$ are projections of products of square-$\kappa^+$-cc and $\kappa^+$-closed forcings.
\end{fact}

We thus see that  the conclusions of Fact \ref{Mitchell-preservation-summary} holds in the case of the guessing variants of Mitchell forcing.

\subsection{A Laver-Like Guessing Sequence for Mahlo Cardinals}

In order to optimize our large cardinal assumptions when using the $\M_\ell(\kappa,\lambda)$ poset, we will use guessing principles analogous to the one introduced by Laver for supercompact cardinals \cite{Laver:i}. Our efforts will extend this framework to Mahlo cardinals, as explained in this section. First we will establish some context.

The following principle was introduced by Hamkins \cite{Hamkins:LaverDiamond} as a weakly compact analog to the original Laver diamond:

\begin{definition} Suppose that $\lambda$ is a weakly compact cardinal. We say that the \emph{weakly compact Laver diamond} holds at $\lambda$ if there is a function $\ell:\lambda \to V_\lambda$ satisfying the following: For any $A\in H(\lambda^+)$ and any transitive structure $M$ of size $\lambda$ with $A,\ell\in M$, there is a transitive set $N$ and an elementary embedding $j:M \to N$ with critical point $\lambda$ so that $j,M\in N$ and so that $j(\ell)(\lambda)=A$.\end{definition}

\begin{fact}\cite{Hamkins:LaverDiamond} If $\lambda$ is weakly compact, there is a forcing extension in which $\lambda$ remains weakly compact and such that the weakly compact Laver diamond holds at $\lambda$.\end{fact}

We will use this fact in concert with $\M_\ell(\kappa,\lambda)$ below in the proofs of Theorems~\ref{tp-csr-notap} and \ref{csr-notsh}.

Now we will explain how we use embeddings with Mahlo cardinals.\footnote{The use of embeddings with Mahlo cardinals has been recently dealt with elsewhere \cite{HLN:small-embedding}.} Typically, we work with a transitive set $N$ modeling a rich fragment of $\ZFC$ which is closed under $\lambda^+$-sequences in $V$. If $M$ is an elementary substructure of $N$, let $\pi_M:M\then \bar{M}$ denote the transitive collapse. For $a\in M$ we denote $\pi_M(a)$ as $\bar{a}_M$ and if $M$ is clear from the context we simply write $\pi$ and $\bar{a}$, respectively.

\begin{definition}\label{def:rich}
We say that $M\prec N$ is \emph{rich} if the following hold: 
\bce[(i)]
\item $\lambda \in M$;
\item $\bar{\lambda}_M:=\pi_M(\lambda)$ is a subset of $M$ (i.e. $M \cap \lambda \in \lambda$);
\item $\bar \lambda_M$ is an inaccessible cardinal in $N$;
\item The size of $M$ is $\bar{\lambda}_M$;
\item $M$ is closed under $<\bar{\lambda}_M$ sequences and $\bar{\lambda}_M<\lambda$.
\ece 
\end{definition}

Note that if $M$ is rich, then $\pi_M^{-1}$ is an elementary embedding from $\bar M$ to $N$ which we will denote $j_M$. The critical point of $j_M$ is $\bar{\lambda}_M$ and $j_M(\bar{\lambda}_M)=\lambda$.  If the $M$ is clear from the context we just write $j$.

If $\lambda$ is Mahlo, then there are plenty of these models:

\begin{fact} \label{P:suitable}
Assume that $\lambda$ is Mahlo. For each $a \subset N$ with $|a|<\lambda$, there is a rich model $M$ such that $a \subset M$.
\end{fact}

We now recall Jensen's \emph{diamond principle} \cite{DEVbook}:

\begin{definition} If $\lambda$ is regular and $S \subset \lambda$ is stationary, then $\diamondsuit_\lambda(S)$ holds if and only if there is a sequence $\seq{X_\alpha}{\alpha \in S}$ such that for all $\alpha \in S$, $X_\alpha \subset \alpha$, and such that for all $Y \subset \lambda$, the set $\set{\alpha \in S}{Y \cap \alpha = X_\alpha}$ is stationary.\end{definition}

It turns out that the principle that we need is in fact equivalent to a diamond principle.

\begin{lemma}\label{mahlolaverfunction} The following are equivalent:

\begin{enumerate}

\item $\lambda$ is Mahlo and $\diamondsuit_\lambda(\textup{Reg})$ holds.

\item  For every transitive structure $N$ satisfying a rich fragment of $\ZFC$ that is closed under $\lambda^+$-sequences in $V$, there is a function $\ell : \lambda \to V_\lambda$ such that the following holds: For every $A \in N$ with $A \in H(\lambda^+)$ and any $a \subset N$ with $|a|<\lambda$, there is a rich $M \prec N$ with $a \cup \{\ell\} \subset M$ such that $\ell(\bar \lambda_M) = \pi_M(A)$.

\end{enumerate}\end{lemma}

\begin{proof} First we prove that \emph{(2)} implies \emph{(1)}. Let $\ell: \lambda \to V_\lambda$ be given. By choosing $A=C$ for any club, it must be the case that if $M$ witnesses the statement then $\bar \lambda_M \in C$ and so $\lambda$ must be Mahlo. Note that to show that $\diamondsuit_\lambda(\textup{Reg})$ holds it is enough to produce a sequence $\seq{X_\delta}{\delta \in \lambda \cap \textup{Reg}}$ satisfying the weakened condition that for any $A \subset \lambda$, there is some regular $\delta<\lambda$ such that $A \cap \delta = X_\delta$. (See Exercise 3A from Section 3 of Chapter III of \cite{DEVbook}.) For all $\delta \in \lambda \cap \textup{Reg}$, let $X_\delta:=\ell(\delta)$ if $\ell(\delta) \subseteq \delta$ and let $X_\delta=\emptyset$ otherwise. Then if $M \prec N$ is the witnessing rich model and $\delta:=M \cap \lambda$, it follows that $\ell(\delta)=\pi_M(A)=A \cap \delta \subseteq \delta$, so $X_\delta = A \cap \delta$.

Now we prove that \emph{(1)} implies \emph{(2)}. Let $\seq{X_\delta}{\delta \in \lambda \cap \textup{Reg}}$ be a $\diamondsuit_\lambda(\textup{Reg})$-sequence and let $\Gamma : \lambda \times \lambda \to \lambda$ be the G{\"o}del pairing function. Let $\ell: \lambda \to V_\lambda$ be defined as follows: If $\delta \in \lambda \cap \textup{Reg}$ and the inverse image of $X_\delta$ under $\Gamma$ is a well-founded extensional relation $E_\delta$ on $\delta$, then let $\ell(\delta)$ be the Mostowski collapse of $E_\delta$. Otherwise we can let $\ell(\delta) = \emptyset$.

Now we will show that $\ell$ works. Let $A \in H(\lambda^+)$ and assume without loss of generality that $|\trcl A| = \lambda$. Choose a bijection $F: \lambda \to \trcl A$, let $E_A \subset \lambda \times \lambda$ code the lattice of $\trcl A$ under $F$, and let $C_A$ be image of $E_A$ under $\Gamma$.

In order to properly use the $\diamondsuit_\lambda(\textup{Reg})$-sequence, we first define a $\sub$-increasing sequence $\seq{M_\xi}{\xi<\lambda}$ of elementary submodels of $N$ by induction on $\xi$. For $\xi=0$, let $M_0$ be an elementary submodel of $N$ of size $<\lambda$ such that $a \cup \{\lambda,\ell,A,F\} \subset M_0$. For $\xi$ limit take the union $\bigcup_{\eta<\xi}M_\eta$. If $\xi=\eta+1$, then take $M_\xi$ to be an elementary submodel of $N$ of size $<\lambda$ such that: $M_\eta^{<|M_\eta|}\sub M_\xi$, $M_\eta\in M_\xi$, and $\mbox{sup}(M_{\eta}\cap\lambda)\sub M_\xi$. For each $\xi$, let $\lambda_\xi$ denote $\mbox{sup}(M_{\xi}\cap\lambda)$. Then $\set{\lambda_\xi}{\xi<\lambda}$ is a closed unbounded subset of $\lambda$, and since $\lambda$ is Mahlo, there is an inaccessible $\xi$ such that $\xi=\lambda_\xi$. It is straightforward to verify that for any such $\xi$, $M_\xi = M$ is a rich submodel of $N$.

Using the fact that $\seq{X_\delta}{\delta \in \lambda \cap \textup{Reg}}$ is a $\diamondsuit_\lambda(\textup{Reg})$-sequence, we can find a rich $M \prec N$ containing $A,F,E_A,C_A$ such that $\delta := M \cap \lambda$ is inaccessible and $X_\delta = C_A \cap \delta$. By elementarity, $F \rest \delta \to (\trcl A)^M$ is a bijection, $E_A \cap \delta \times \delta$ codes the lattice of $(\trcl A)^M$, and $C_A \cap \delta$ is the image of $E_A$ under $\Gamma$. It follows that $\ell(\delta) = \pi(E_A \cap (\delta \times \delta)) = \pi_M(A)$.\end{proof}

This formulation of $\diamondsuit_\lambda(\textup{Reg})$ in terms of rich models will be used in the proofs of Theorems~\ref{gk} and \ref{sr-notsh} below.

\section{Results from Weakly Compact Cardinals}\label{sec-wc-results}

\subsection{Obtaining $\CSR(\kappa^{++}) \wedge \TP(\kappa^{++}) \wedge \AP(\kappa^{++})$}\label{sec-tp-csr-ap}

In this section we prove Theorem~\ref{tp-csr-ap}. Assume that $\kappa^{<\kappa}=\kappa$, $\lambda$ is weakly compact, and $\kappa<\lambda$. Let us denote the Mitchell forcing $\M(\kappa,\lambda)$ from Definition~\ref{Def:Mitchell} as $\M$. Without loss of generality, $2^\lambda = {\lambda^+}$.

Working in $V[\M]$, we define a standard club-adding iteration of length $\lambda^+$: $\P_{\lambda^+}:=\seq{\P_\alpha,\dot \Q_\alpha}{\alpha < {\lambda^+}}$. Let $F: {\lambda^+} \to {\lambda^+} \times {\lambda^+}$ be a bijection such that if $F(\alpha)=(\beta,\gamma)$ then $\beta \le \alpha$. By induction we define both the iteration and sequences $\seq{\dot S_\alpha}{\alpha<{\lambda^+}}$ and $\seq{\dot T_\alpha}{\alpha<{\lambda^+}}$. If $\P_\beta$ is defined, we set $\dot{S}_\alpha$ to be the $\gamma^{th}$ nice $\P_\beta$-name for a stationary subset of $\lam$, where $\alpha=F(\beta,\gamma)$, and we set $\dot T'_\alpha$ to be a canonically defined nice $\P_\beta$-name (using the mixing of $\P_\beta$-names) such that $\Vdash_{\P_\beta} \text{``}\dot T'_\alpha = \{ \rho \in \lambda \cap \cof(\kappa^+): \dot S_\alpha \text{ reflects in }\rho\}$''. Then we let $\dot T_\alpha$ be the canonical $\P_\beta$-name for $\dot T'_\alpha \cup (\lambda \cap \cof(\le \kappa))$. Since the supports at limit stages are given by Definition \ref{def:SCAI}, to complete the definition, it suffices to define $\P_{\beta+1}$ given $\P_\beta$: let $\alpha=F(\beta,\gamma)$, and let $\dot{\Q}_\beta$ be the $\P_\beta$-name for $\textup{CU}(\dot{S}_\alpha)$. Then define $\P_{\beta+1}=\P_\beta\ast\dot{\Q}_\alpha$

\begin{proposition} $\P_{\lambda^+}$ has the $\lambda^+$-chain condition.\end{proposition}

Thus $\P_{\lambda^+}$ preserves cardinals $\ge \lambda^+$. We also get some closure for $\P_\alpha$, $\alpha \le \lambda^+$ which we will need in order to apply the Absorption Lemma.

\begin{proposition}\label{closure} For all $\alpha \le \lambda^+$, $\P_\alpha$ is $\kappa^+$-closed.\end{proposition}

\begin{proof} For all $\beta<\alpha$ we claim that $\Vdash_{\P_\beta} \text{``}\textup{CU}(\dot T_\beta)$ is $\kappa^+$-closed'': if $\seq{c_\xi}{\xi<\tau} \subset \textup{CU}(T_\beta)$ is strictly decreasing and $\tau<\kappa^+$, then $\sup \set{\max c_\xi}{\xi<\tau}$ has cofinality $\cf(\tau)\le\tau<\kappa^+$, and so $\bigcup_{\xi<\tau}c_\xi \cup \{\sup\set{\max c_\xi}{\xi<\tau} \}$ is vacuously a condition in $\textup{CU}(T_\beta)$. Hence the iteration $\P_\alpha$ is $\kappa^+$-closed since $\kappa^+ < \lambda$ and the supports of the conditions are sufficiently large. \end{proof}

\begin{lemma}\label{mainlifting}  Fix $G$, an $\M$-generic over $V$. Suppose that $\alpha < {\lambda^+}$ and that $M$ is a $\lambda$-sized transitive model such that $\M \ast \P_\alpha \in M$. Suppose also that $N$ and $j:M \to N$ witness the weak compactness of $\lambda$. Then the following are true:

\begin{enumerate}
\item\label{mainlifting-embedding} Let  $H_\alpha$ be any $\P_\alpha$-generic over $V[G]$. There is $\tilde G$, a $j(\M)$-generic over $V$ and $\tilde H_\alpha$, a $j(\P_\alpha)$-generic over $V[\tilde G]$ such that $N[G \ast H_\alpha] \subset N[\tilde G \ast \tilde H_\alpha]$ and such that $V[\tilde G \ast \tilde H_\alpha]$ defines a lift $j:M[G \ast H_\alpha] \to N[\tilde G \ast \tilde H_\alpha]$. Moreover, the extension of $N[\tilde G \ast \tilde H_\alpha]$ over $N[G \ast H_\alpha]$ does not add cofinal branches to $\lambda$-trees from $N[G \ast H_\alpha]$.

\item\label{mainlifting-distributivity} $M[G] \models \text{``}\P_\alpha$ is $\lambda$-distributive''.

\end{enumerate}\end{lemma}

\begin{proof} We argue for any $\alpha<\lambda^+$ and prove (\ref{mainlifting-embedding}) and (\ref{mainlifting-distributivity}) for $\alpha$ at the same time. Specifically, we prove (\ref{mainlifting-embedding}) by constructing a lift in two steps, and part of the proof will show that $M[G][H_\alpha]^{<\lambda} \subset M[G]$, so (\ref{mainlifting-distributivity}) will follow because $H_\alpha$ is arbitrary.

First, we lift the embedding to $M[G]$. Now work in $V[G]$. Since $\P_\alpha$ is $\kappa^+$-closed by Lemma~\ref{closure} and has size $\lambda$, the Absorption Lemma for Mitchell Forcing implies that $j(\M)/G \times \P_\alpha \simeq j(\M)/G$. Let $G'$ be a $j(\M)/G$-generic over $V[G][H_\alpha]$; then there is a $j(\M)/G$-generic $G''$ over $V[G]$ such that $V[G][H_\alpha][G']=V[G][G'']$. Let us denote $\tilde G=G\times G''$.

Since $j$ is the identity below $\lambda$ and conditions in $\M(\kappa,\lambda)$ are bounded below $\lambda$, $j[G]=G \subset G\times G''=\tilde G$, and so we can apply Fact~\ref{weakliftinglemma} to lift $j:M[G] \to N[\tilde G]$. Now we perform the second step of the lift to extend $j$ to have domain $M[G][H_\alpha]$.

\begin{Claim}\label{mainlifting-mastercondition} $N[\tilde G]$ contains a strong master condition for $H_\alpha$ with respect to $j:M[G] \to N[\tilde{G}]$.\end{Claim}

\begin{proof}[Proof of Claim]  We argue, working in $N[\tilde G]$, that there is $q \in j(\P_\alpha)$ that is a lower bound for $j[H_\alpha]$: Let $\seq{C_\beta}{\beta<\alpha}$ be the sequence of clubs added by $H_\alpha$, where $C_\beta$ is added at the $(\beta+1)^\text{st}$ step of the iteration. Observe that the $C_\beta$'s are closed and unbounded in $\lambda$, the latter because $\lambda \cap \cof(\le\!\kappa)$ is unbounded in $\lambda$. Let $\dom q = j[\alpha]$ and for $\beta<\alpha$, let $q(j(\beta)) = C_\beta \cup \{\lambda\}$, noting that $q\in N[\tilde G]$ since $j$ and $H_\alpha$ are in $N[\tilde G]$. It will be immediate that $q$ is a lower bound of $j[H_\alpha]$ once we verify that for $\beta<\alpha$, $q\rest j(\beta)\Vdash^{N[\tilde{G}]}_{j(\P_\beta)} \text{``}C_\beta \cup \{\lambda\}\text{ is a condition in }j(\CU(\dot T_\beta))$''.

First note that for all $\beta<\lambda$, $q\rest j(\beta) \Vdash^{N[\tilde{G}]}_{j(\P_\beta)}$ ``$C_\beta$ is a closed unbounded set in $j(\dot{T}_\beta)\cap\lambda$'' since  $q\rest j(\beta) \Vdash^{N[\tilde{G}]}_{j(\P_\beta)}$ ``$j(\dot{T}_\beta)\cap\lambda=T_\beta$; therefore to verify that $q$ is a condition in $j(\P_\alpha)$ it suffices to show that that $q\rest j(\beta)\Vdash^{N[\tilde{G}]}_{j(\P_\beta)} \text{``}j(\dot S_\beta) \cap \lambda \text{ is stationary in }\lambda$''. 

Since $q\rest j(\beta)\Vdash^{N[\tilde{G}]}_{j(\P_\beta)} $``$j(\dot S_\beta) \cap \lambda = S_\beta$, it remains to argue that $N[\tilde G] \models \text{``}S_\beta$ is stationary after forcing with $j(\P_\beta)$''. First note that $S_\beta$ is still stationary in $N[G][H_\alpha]$ by Fact~\ref{F:St_Closed_cof} because $N[G][H_\alpha]$ is a generic extension of $N[G][H_\beta]$ by a $\kappa^+$-closed forcing.\footnote{We did not assume distributivity of $\P_\beta$ and hence preservation of $\lambda$ by induction, but one can observe that the $\kappa^+$-closure would preserve stationarity anyway if $\lambda$ were collapsed to have cofinality $\kappa^+$.} Second, $S_\beta$ is stationary after forcing with $(j(\M)/G)^{N[G]}$ (by Lemma~\ref{preservation}), so it is stationary in $N[G][H_\alpha][G']=N[\tilde G]$. Furthermore, by elementarity of the embedding $j:M[G] \to N[\tilde G]$, we have $N[\tilde G] \models \text{``}j(\P_\beta)$ is $\kappa^+$-closed'', so $ N[\tilde G] \models \text{``}j(\P_\beta)$ preserves the stationarity of $S_\beta$'' by Fact~\ref{F:St_Closed_cof}.\end{proof}

Let $\tilde H_\alpha$ be any $j(\P_\alpha)$-generic over $N[\tilde G]$ that contains the master condition $q$ from Claim~\ref{mainlifting-mastercondition}. This tells us that $j[H_\alpha] \subset \tilde H_\alpha$, that we can extend to $j:M[G][ H_\alpha] \to N[\tilde G][ \tilde H_\alpha]$ by Fact~\ref{weakliftinglemma}, and that $M[G]$ and $M[G][H_\alpha]$ have the same $<\lambda$-sequences of ordinals by Fact~\ref{strongliftinglemma}. Hence we have (\ref{mainlifting-distributivity}) for $\alpha$, as explained above.

To finish the proof it suffices to show that cofinal branches are not added to $\lambda$-trees in the extension of $N[\tilde G \ast \tilde H_\alpha]$ over $N[G \ast H_\alpha]$. If $T$ is a $\lambda$-tree in $N[G][H_\alpha]$, then there are no new cofinal branches in $N[G][H_\alpha][G']=N[\tilde G]$ by Lemma~\ref{preservation}. Note that in $N[\tilde G]$, $\lambda$ is an ordinal of cofinality $\kappa^+$, and therefore the tree $T$ has height of cofinality $\kappa^+$ and levels of size $\kappa^+$. However, Fact \ref{silverslemma} still applies since $2^\kappa>\kappa^+$ in $N[\tilde G]$ and $j(\P_\alpha)$ is $\kappa^+$-closed in $N[\tilde G]$ by elementarity, and we are finished.
\end{proof}

\begin{lemma}\label{wcdist} $V[\M] \models \text{``}\P_{\lambda^+}$ is $\lambda$-distributive''.\end{lemma}

\begin{proof} By Fact~\ref{nicenamestrick} it is enough to show that $V[\M] \models \text{``}\P_\alpha$ is $\lambda$-distributive'' for all $\alpha < \lambda^+$. Moreover, by choosing a supposed counterexample $\dot f$, it is enough to prove that $M[\M] \models \text{``}\P_\alpha$ is $\lambda$-distributive'' for a $\lambda$-sized transitive $M$ with $\dot f$ and $\P_\alpha$ in  $M$, which follows from Lemma~\ref{mainlifting}(\ref{mainlifting-distributivity}).\end{proof}

\begin{lemma} $V[\M \ast \P_{\lambda^+}] \models \CSR(\lambda)$. \end{lemma}

\begin{proof} If $\dot S$ is a $\P_{\lambda^+}$-name for a stationary subset of $\lambda \cap \cof(\le\! \kappa)$, then Fact~\ref{nicenamestrick} implies that without loss of generality, there is some $\alpha' < \lambda^+$ such that $\dot S$ is a $\P_{\alpha'}$-name. By construction, it follows that for some $\beta,\gamma<\lambda^+$, $\dot S$ is the $\gamma^{th}$ $\P_\beta$-name for such a stationary set, so $\dot S = \dot S_\alpha$ for some $\alpha < \lambda^+$. As observed in the proof of Claim~\ref{mainlifting-mastercondition}, the generic club added by the next step of the iteration is unbounded, so $V[\M \ast \P_{\alpha+1}] \models$ ``There is a club $C \subset \lambda$ such that every $\rho \in C \cap \cof(\kappa^+)$ is a reflection point of $S$''. Then $V[\M \ast \P_{\lambda^+}]$ still satisfies this statement by distributivity of $\P_{\lambda^+}$ over $V[\M]$, and hence distributivity of the quotient over $V[\M \ast \P_{\alpha+1}]$.\end{proof}

\begin{lemma}\label{wc-tp} $V[\M \ast \P_{\lambda^+}] \models \TP(\lambda)$.\end{lemma}

\begin{proof} By the $\lambda^+$-chain condition, the name for any supposed $\lambda$-Aronszajn tree $T$ in $V[\M \ast \P_{\lambda^+}]$ would be contained in $V[\M\ast\P_\alpha]$ for some $\alpha<\lambda^+$, and so $T$ would be contained in $M[\M \ast \P_\alpha]$ for some $\lambda$-sized transitive model $M$. Let $j:M \to N$ witnesses weak compactness. By Lemma~\ref{mainlifting}, if we fix generics $G$ and $H_\alpha$, we can lift the embedding to $j:M[G][ H_\alpha] \to N[\tilde  G][ \tilde H_\alpha]$. So $N[\tilde G][ \tilde H_\alpha] \models \text{``}j(T)$ is a $j(\lambda)$-tree'' by elementarity, and any element $t \in j(T)_\lambda$ defines a cofinal branch $\{s \in T: s \le_{j(T)} t \}$ in $T$. Because of the preservation properties in Lemma~\ref{mainlifting}(\ref{mainlifting-embedding}), this branch is already contained in $N[G][ H_\alpha]$, so it is contained in $V[G][ H_\alpha]$.\end{proof}

\subsection{Obtaining $\CSR(\kappa^{++}) \wedge \TP(\kappa^{++}) \wedge \neg\AP(\kappa^{++})$}\label{sec-tp-csr-notap}

For this section we work in a model $V$ in which $\lambda$ is weakly compact and carries a weakly compact Laver diamond. Fix a witnessing function $\ell:\lambda \to V_\lambda$. Let $\M:=\M_\ell(\kappa,\lambda)$ be the guessing version of the Mitchell forcing from Definition~\ref{def:myMitchell} defined in terms of this $\ell$.

We will also assume that $\kappa<\lambda$ is regular and that $\kappa^{<\kappa}=\kappa$. Let $\P$ be the club-adding iteration defined as in Section~\ref{sec-tp-csr-ap}. As before, $\P$ is $\lambda^+$-cc. Our main task is to show that $\P$ is $\lambda$-distributive while showing that we can lift embeddings as necessary.

\begin{lemma}\label{mainlifting-notap}  Fix $G$, an $\M$-generic over $V$. Suppose that $\alpha < {\lambda^+}$ and that $M$ is a $\lambda$-sized transitive model such that $\M \ast \P_\alpha \in M$. Suppose also that $N$ and $j:M \to N$ witness weak compactness of $\lambda$, and moreover that $j(\ell)(\lambda) = \dot{\P}_\alpha$. Then the following hold:

\begin{enumerate}
\item Let $H_\alpha$ be any $\P_\alpha$-generic over $V[G]$. There are $\tilde G$, a $j(\M)$-generic over $V$, and $\tilde H_\alpha$, a $j(\P_\alpha)$-generic over $V[\tilde G]$, such that $N[G \ast H_\alpha] \subset N[\tilde G \ast \tilde H_\alpha]$ and such that $V[\tilde G \ast \tilde H_\alpha]$ defines a lift $j:M[G \ast H_\alpha] \to N[\tilde G \ast \tilde H_\alpha]$. Moreover, the extension of $N[\tilde G \ast \tilde H_\alpha]$ over $N[G \ast H_\alpha]$ does not add cofinal branches to $\lambda$-trees from $N[G \ast H_\alpha]$.

\item $M[G] \models \text{``}\P_\alpha$ is $\lambda$-distributive''.

\end{enumerate}\end{lemma}

\begin{proof} Instead of going through the whole argument, we will explain here how the proof of Lemma~\ref{mainlifting} can be modified for our present purposes. The crux is as follows: Starting from $V[G][H_\alpha]$, we must find a $j(\M)$-generic $\tilde G$ such that $V[G][H_\alpha][I][G']=V[\tilde{G}]$ where $I$ is $\Add(\kappa,\lambda^*)$-generic over $V[G][H_\alpha]$ ($\lambda^*$ being the least inaccessible in $N$ above $\lambda$) and $G'$ is $j(\M)/(G *(H_\alpha \times I))$-generic over $V[G][H_\alpha][I]$. However, instead of using the Absorption Lemma, we use the guessing properties of $\ell$.

To see how we can obtain such an $G'$, first observe that $\dot{\P}_\alpha=j(\ell)(\lambda)$ is a $j(\M)\rest\lambda=\M$-name for a $\kappa^+$-closed poset. Thus,  $j(\M)$ is isomorphic to a dense subset of
$$
j(\M)\rest\lambda^*\ast  \dot{\N}_{\lambda^*}\cong \M\ast(\dot{\P}_\alpha\times\Add(\kappa,\lambda^*))\ast  \dot{\N}_{\lambda^*}.
$$ 

Hence we can choose any $I$ that is $\Add(\kappa,\lambda^*)$-generic over $V[G][H_\alpha]$ and any $G'$ that is $\N_{\lambda^*}$-generic over $V[G][H_\alpha][I]$, and we let $\tilde G := G \ast (H_\alpha \times I) \ast G'$.

\begin{Claim}\label{mainlifting-notap-mastercondition} $N[\tilde G]$ contains a strong master condition for $H_\alpha$ with respect to $j:M[G] \to N[\tilde G]$.\end{Claim}

As in Claim~\ref{mainlifting-mastercondition}, the main point is to inductively show that for $\beta<\alpha$, the $S_\beta$'s remain stationary after forcing with the master condition over $N[\tilde G]$. The argument here is almost identical, except that we now require an application of Fact \ref{F:St_cc} to show that $S_\beta$ remains stationary in the extension from $N[G][H_\alpha]$ to $N[G][H_\alpha \times I]$. Then the $G'$, since it is again a projection of a product of a $\kappa^+$-cc and $\kappa^+$-closed poset, acts analogously to the $G'$ from the proof of Lemma~\ref{mainlifting}. The last step using the $\kappa^+$-closure of $j(\P_\alpha)$ is also the same.

Moving forward, we choose $\tilde H_\alpha$ to be a generic containing the master condition to define the extension $N[\tilde G \ast \tilde H_\alpha] = N[G][H_\alpha][I][G'][\tilde H_\alpha]$. The last way in which this argument differs from that for Lemma~\ref{mainlifting} is in showing that the extension from $N[G \ast H_\alpha]$ to $N[\tilde G \ast \tilde H_\alpha]$ does not add cofinal branches to $\lambda$-trees. In our present case we must again pay attention to the extension $N[G \ast H_\alpha][I]$ over $N[G \ast H_\alpha]$, so we use an extra application of Fact \ref{F:Square_cc}.\end{proof}

It follows from our lifting lemma that we have the key properties for our model $V[\M \ast \P]$. The arguments for the following differ from those in Section~\ref{sec-tp-csr-ap} only insofar as we must use the weakly compact Laver diamond to \emph{specifically} choose a weakly compact embedding $j:M \to N$ such that $j(\ell)(\lambda) = \dot \P_\alpha$.

\begin{lemma} $V[\M] \models $``$\P$ is $\lambda$-distributive''.\end{lemma}

\begin{lemma} $V[\M \ast \P] \models \CSR(\kappa^{++})$.\end{lemma}

\begin{lemma} $V[\M \ast \P] \models \TP(\kappa^{++})$.\end{lemma}

It remains to obtain $\neg \AP(\kappa^{++})$. To do so, we will use the lifting lemma and mimic the proof from \cite{8fold}.

\begin{lemma}\label{wc-notap} $V[\M \ast \P] \models \neg \AP(\kappa^{++})$.\end{lemma}

\begin{proof} Suppose for contradiction that $\AP(\kappa^{++})$ holds in some extension by $\M\ast\dot{\P}$. We may fix an enumeration $\vec{a}=\la a_i:i<\lambda\ra \in V[\M \ast \dot{\P}_{\lambda^+}]$ of bounded subsets of $\lambda$ and find a club $C\subseteq\lambda$ in $V[\M \ast \dot{\P}_{\lambda^+}]$ so that every $\gamma\in C$ is approachable with regard to $\vec{a}$. By Fact~\ref{nicenamestrick}, there is some $\alpha<\lambda^+$ such that $\vec{a}, C \in V[\M \ast \dot{\P}_\alpha]$.

Next, we will make use of Lemma~\ref{mainlifting-notap}. Let $M$ be a transitive $\lambda$-sized model containing $\P_\alpha, \ell, \vec{a}$, and $C$. Since $\ell$ witnesses the weakly compact Laver diamond for $\lambda$, we find $N$ such that $j:M \to N$ witnesses weak compactness of $\lambda$ and $j(\ell)(\lambda) = \dot{\P}_\alpha$. Use Lemma~\ref{mainlifting-notap} to obtain an embedding $j:M[G\ast H_\alpha]\to N[\tilde G\ast \tilde H_\alpha]$ for some $N$, where $G\ast H_\alpha$ is $\M\ast\dot{\P}_\alpha$-generic over $V$. Note that $j(C)\cap\lambda=C$ and that $j(\vec{a})\rest\lambda=\vec{a}$. Consequently, $\lambda$ is a limit point of $j(C)$ and hence a member of $j(C)$, and by definition of $j(C)$, we have that $\lambda$ is approachable with respect to $\vec{a}$ in $N[\tilde G\ast \tilde H_\alpha]$. Since $N[\tilde G \ast \tilde{H}_\alpha] \models $``$\cf(\lambda)=\kappa^+$'', we find a club $e\subseteq\lambda$ of order-type $\kappa^+$ witnessing this. By definition, we then have that every initial segment of $e$ is in the range of $\vec{a}$, and hence every initial segment of $e$ is a member of $M[G\ast H_\alpha]$.

Let $U$ denote the tree $(\lambda^{<\kappa^+})^{V[G\ast H_\alpha]}$, and observe that $U$ has width $\lambda^\kappa=\lambda$. By the remarks in the previous paragraph, we see that $e$ gives a cofinal branch through $U$. Moreover, $e\in N[\tilde G \ast {\tilde H}_\alpha]$. Note that $V[G \ast H_\alpha] \models $``$\lambda$ is regular'', and this remains true in $V[G \ast H_\alpha \ast I]$ where $I$ is the $\Add(\kappa,\lambda^*)$-generic from Lemma~\ref{mainlifting-notap}. Therefore, if we show that $e \in V[G \ast H_\alpha \ast I]$, we obtain a contradiction since the order-type of $e$ is $\kappa^+$.

Now we argue that $e \in V[G \ast H_\alpha \ast I]$. By the $j(\lambda)$-distributivity of $j(\P_\alpha)$, we conclude that $e \in N[\tilde G]$. However, with the same notation as in Lemma~\ref{mainlifting-notap}, we see that in the model $N[G \ast H_\alpha\ast I] \subset N[G \ast H_\alpha\ast I \ast H'] = N[\tilde G]$, we have that $2^\kappa>\lambda$, meaning that $2^\kappa$ is greater than the width of $U$. Hence adjoining $H'$ does not add a branch through $U$ because in $N[G \ast H_\alpha\ast I]$,  $\N_{\lambda^*}$ is a projection of a product of a $\kappa^+$-closed forcing and a square-$\kappa^+$-cc forcing. Therefore $e$ is a member of $V[G\ast H_\alpha \ast I]$, a contradiction.\end{proof}

\subsection{Obtaining $\CSR(\kappa^{++}) \wedge \neg\SH(\kappa^{++}) \wedge \neg \AP(\kappa^{++})$}

Now we will prove Theorem~\ref{csr-notsh}. Starting from a model in which $\lambda$ is weakly compact and which contains an $\ell$ witnessing the weakly compact Laver diamond, we will force with the Mitchell-type guessing poset $\M$ from Definition~\ref{def:myMitchell}, then with Kunen's poset $\S$ for adding a $\lambda$-Suslin tree, and finally with the Magidor-style iteration $\P$ for obtaining $\CSR(\kappa^{++})$. The technical challenge will be to show not only that $\P$ is $\lambda$-distributive, but also that if $T$ is the generic tree added by $\S$, then $T$ remains Suslin after forcing with $\P$. Since $\P$ will not be $\lambda$-closed, an appeal to Easton's Lemma will not work. Instead, the proofs of both facts will be done simultaneously by an inductive argument on the length of $\P$ (like the ones in Sections \ref{sec-tp-csr-ap} and \ref{sec-tp-csr-notap}) in which we build a master condition to show that a weakly compact embedding can be lifted. Here the added complexity will consist of the fact that the lifting argument and the argument that $T$ remains Suslin will rely on one another. In particular, building a master condition will require us to know that the image of $T$ under the embedding has a cofinal branch, and this will be achieved by forcing to deliberately add such a branch.

Before commencing with the proof, we review some standard facts about Kunen's forcing to add a Suslin tree by initial segments \cite{KUNENsat}.

\begin{definition} Let $\S$ consist of all trees $t\subseteq\,^{<\lambda}2$ of successor height $\alpha+1$, for some $\alpha<\lambda$, which satisfy the following:
\begin{enumerate}
\item $t$ is a normal tree, and all levels of $t$ have size $<\lambda$;
\item $t$ is homogeneous.
\end{enumerate}
The ordering on $\S$ is end-extension.
\end{definition}

We recall that $t$ is \emph{homogeneous} if for any $s\in t$, the subtree of $t$ above $s$ equals $t$. In precise notation, $t$ is homogeneous if and only if for all $s\in t$, $t_s=t$, where $t_s=\{ u:s^\frown u\in t\}$. Note that this is equivalent to saying that for any two sequences $s,u$, we have that $s^\frown u\in t$ if and only if both $s$ and $u$ are in $t$.

We recall the following standard fact about this forcing:

\begin{fact}\label{lemma:branchextend} Suppose that $\gamma<\lambda$ is a limit, that $t$ is a normal, homogeneous $\gamma$-tree with all levels of size $<\lambda$, and that $b\subseteq t$ is a cofinal branch. Then there exists a condition $t'\in\S$ of height $\gamma+1$ extending $t$ so that $b \in t'$.
\end{fact}

More precisely, we define
$$
t':=t\cup\{ s^\frown(b\setminus\alpha):s\in t\wedge \alpha<\gamma\},
$$
where for each $\alpha<\gamma$, $b \setminus \alpha$ denotes the unique $u$ so that $(b\rest\alpha)^\frown u=b$, i.e., the tail segment of $b$ above $\alpha$. It is straightforward to check that $t'$ is a condition in $\S$ which extends $t$. The condition $t'$ constructed in the proof of the above lemma will be known as \emph{the minimal extension of $t$ by $b$}. 

A crucial tool in analyzing $\S$ is to work with the two-step forcing of $\S$ followed by forcing with the generic tree. We will use $\dot{T}$ to denote the $\S$-name for the generic tree.

\begin{fact} Let $\dot{\T}$ be the $\S$-name for the forcing $(\dot{T},<_{\dot{T}})$. Then:

\begin{enumerate}
\item $\S$ is $\lambda$-strategically closed and hence $\lambda$-distributive.
\item $\S\ast\dot{\T}$ has a dense, $\lambda$-closed subset.
\item $\S\ast\dot{\T}$ is forcing equivalent to $\Add(\lambda,1)$.
\end{enumerate}
\end{fact}

Now let us commence with the proof. Let $\kappa,\lambda$ be as in Sections \ref{sec-tp-csr-ap} and \ref{sec-tp-csr-notap}, let $\M:=\M_\ell(\kappa,\lambda)$ be the guessing variant of the Mitchell forcing, let $\S$ be Kunen's forcing for adding a $\lambda$-Suslin tree as defined in $V[\M]$, and let $\P$ be the Magidor-style iteration from Sections \ref{sec-tp-csr-ap} and \ref{sec-tp-csr-notap}, this time defined in $V[\M \ast \S]$. Then $V[\M \ast \S \ast \P]$ will be our intended model for Theorem~\ref{csr-notsh}.

\begin{proposition} $V[\M] \models $``$\bb{K}_\alpha:=\S \ast(\dot{\T}\times\dot{\P}_\alpha)$ has a $\kappa^+$-closed, dense subset''.\end{proposition}

We will abuse notation and use $\bb{K}_\alpha$ to also denote this dense subset.

\begin{proof} We know that in $V[\M]$, $\S\ast\dot{\T}$ is $\kappa^+$-closed (in fact, $\lambda$-closed) on a dense subset. Furthermore, $\T$ is $\kappa^+$-distributive after forcing with $\S$, so $\P_\alpha$, which is $\kappa^+$-closed after forcing with $\S$, remains $\kappa^+$-closed after forcing with $\S\ast\dot{\T}$. Thus $\bb{K}_\alpha \simeq \S\ast\dot{\T}\ast\dot{\P}_\alpha$ is $\kappa^+$-closed on a dense subset.\end{proof}

As in Sections \ref{sec-tp-csr-ap} and \ref{sec-tp-csr-notap}, all cardinals $\leq\kappa^+$ are easily seen to be preserved. Furthermore, since $\lambda^{<\lambda}=\lambda$ in the extension by $\M$, and since $\S$ is $\lambda$-distributive in the $\M$-extension, the equation $\lambda^{<\lambda}=\lambda$ still holds after forcing with $\S$. Consequently, we may see that $\P$ is $\lambda^+$-cc. Thus all cardinals $\mu\geq\lambda^+$ are preserved. Furthermore, once we know that $\lambda$ is preserved, the $\lambda^+$-cc of $\P$ implies that we may catch our tail and achieve a model in which $\CSR(\lambda)$ holds.

The technical core of the theorem consists of simultaneously showing that $\P$ is $\lambda$-distributive, that $T$ is Suslin in $V[\M \ast \S \ast \P]$, and that we can lift certain weakly compact embeddings.

\begin{lemma}\label{mainlifting-wcnotsh}  Fix $G \ast T$, an $\M \ast \S$-generic over $V$. Suppose that $\alpha < {\lambda^+}$ and that $M$ is a $\lambda$-sized transitive model such that $\ell, \M \ast \S \ast \P_\alpha \in M$. Suppose also that $N$ and $j:M \to N$ witness weak compactness of $\lambda$ and that $j(\ell)(\lambda)=\dot{\bb{K}}_\alpha$. Then the following are true:

\begin{enumerate}
\item\label{mainlifting-wcnotsh-embedding} Let  $H_\alpha$ be any $\P_\alpha$-generic over $V[G \ast T]$. There is a $j(\M \ast \S \ast \P_\alpha)$-generic $\tilde G \ast \tilde T \ast \tilde H_\alpha$ over $V$ such that $N[G \ast T \ast H_\alpha] \subset N[\tilde G \ast \tilde T \ast \tilde H_\alpha]$ and such that $V[\tilde G \ast \tilde T \ast \tilde H_\alpha]$ defines a lift $j:M[G \ast T \ast H_\alpha] \to N[\tilde G \ast \tilde T \ast \tilde H_\alpha]$.

\item\label{mainlifting-wcnotsh-distributivity} $M[G][T] \models \text{``}\P_\alpha$ is $\lambda$-distributive''.

\item\label{mainlifting-wcnotsh-suslinpreservation} $M[G][T] \models $``$T$ remains Suslin after forcing with $\P_\alpha$''.

\end{enumerate}\end{lemma}

\begin{proof} Assuming (\ref{mainlifting-wcnotsh-suslinpreservation}) for $\beta<\alpha$, we prove (\ref{mainlifting-wcnotsh-embedding}) and (\ref{mainlifting-wcnotsh-distributivity}) for $\alpha$ at the same time. Specifically, part of the proof will show that $M[G][T][H_\alpha]^{<\lambda} \subset M[G][T]$, so (\ref{mainlifting-wcnotsh-distributivity}) follows because $H_\alpha$ is arbitrary. Then we will use the type of lift constructed in (\ref{mainlifting-wcnotsh-embedding}) for $\alpha$ to prove (\ref{mainlifting-wcnotsh-suslinpreservation}) for $\alpha$.

\begin{proof}[Proof of (\ref{mainlifting-wcnotsh-embedding}) and (\ref{mainlifting-wcnotsh-distributivity})] We will lift the embedding $j:M \to N$ in three steps.

Our first step is to lift $j:M \to N$ to have domain $M[G]$. Since $j(\ell)(\lambda)=\dot{\bb{K}}_\alpha$ is a $j(\M)\rest\kappa=\M$-name for a $\kappa^+$-closed poset, there is a dense subset of $j(\M)$ giving us the forcing-equivalence

$$
j(\M)\simeq \M\ast(\dot{\bb{K}}_\alpha \times\Add(\kappa,\lambda^*))\ast\dot{\bb{N}}_{\lambda^*},
$$

\noindent where $\lambda^*$ is the least $N$-inaccessible above $\lambda$. Fix a $(\T\times\P_\alpha)$-generic filter $B\times H_\alpha$ over $V[G\ast T]$. Here $B$ denotes the generic branch through $T$. Let $I$ be $\Add(\kappa,\lambda^*)$-generic over $V[G \ast T \ast (B \times H_\alpha)]$, so that $T \ast (B \times H_\alpha \times I)$ is $\bb{K}_\alpha\ast\Add(\kappa,\lambda^*)$-generic over $V[G]$. Let $G'$ be $\N_{\lambda^*}$-generic over $V[G \ast T \ast ((B \times H_\alpha) \times I)]$. Since $j[G]=G\subseteq \tilde G$, we may lift $j$ to an elementary embedding $j:M[G]\to N[\tilde G]$ working in $N[G \ast T \ast ((B \times H_\alpha) \times I) \ast G']=N[j(G)]=N[\tilde G]$.

The next step is to lift the embedding $j$ to have domain $M[G \ast T]$. We claim that there is a condition $t^*\in j(\S)$ so that $T\subseteq t^*$. Indeed, $T$ is a normal, homogeneous $\lambda$-tree in $N[\tilde G]$ all of whose levels have size $<j(\lambda)$. Since $B$ is a cofinal branch through $T$ in $N[\tilde G]$, we may find, by Fact \ref{lemma:branchextend}, a condition in $j(\S)$ which extends $T$. Set $t^*\in j(\S)$ to be the minimal extension of $T$ by $B$. Let $\tilde T$ be $j(\S)$-generic over $V[\tilde G]$ so that $t^*\in \tilde T$. Since $T=j[T]$ and $T$ is an initial segment of $t^*$, we may extend $j$ to an elementary embedding $j:M[G\ast T]\to N[\tilde G\ast \tilde T]$.

The third step in the construction of the embedding is to show that we can lift the embedding $j$ to have domain $M[G \ast T \ast H_\alpha]$.

\begin{Claim} $N[\tilde G \ast \tilde T]$ contains a strong master condition for $H_\alpha$ with respect to $j:M[G \ast T] \to N[\tilde G \ast \tilde T]$.\end{Claim}

\begin{proof}[Proof of Claim] This is analogous to the proof of Claims \ref{mainlifting-mastercondition} and \ref{mainlifting-notap-mastercondition}; instead of going through all the details, we will point to the differences. Once again, the main point is to prove that if $\beta<\alpha$, then $S_\beta$ remains stationary in $N[\tilde G\ast \tilde T]$. By the inductive assumption that (\ref{mainlifting-wcnotsh-suslinpreservation}) holds for $\beta<\alpha$, $T$ is still Suslin after forcing with $\P_\beta$, and hence $\T$ is still $\lambda$-cc in $V[G\ast T\ast H_\beta]$. Thus $S_\beta$ is stationary in $V[G\ast T\ast (B\times H_\beta)]$ (here $H_\beta:=H_\alpha\cap\P_\beta$). The tail of the forcing $\P_\alpha$ from stage $\beta$ onwards is still $\kappa^+$-closed after forcing with $\T$, since $\T$ is $\lambda$-distributive. Since $\kappa^{<\kappa}=\kappa$ in $V[G\ast T\ast H_\beta]$ and $S_\beta$ consists of points of cofinality $\leq\kappa$, $S_\beta$ is stationary in $V[G\ast T \ast B \ast H_\alpha \ast I]$ and hence in $N[G\ast T \ast B \ast H_\alpha \ast I]$ too. The argument that $S_\beta$ remains stationary in $N[\tilde G]$ (after adjoining $G'$) is analogous to the previous ones. Finally, we use the $j(\lambda)$-distributivity of $j(\S)$ to show that $S_\beta$ remains stationary in $N[\tilde G \ast \tilde T]$.\end{proof}

As in the proof of Lemma~\ref{mainlifting}, we choose $\tilde H_\alpha$ to contain the master condition to complete the lift and to obtain $M[G \ast T \ast H_\alpha]^{<\lambda} \subset M[G \ast T]$ and hence (\ref{mainlifting-wcnotsh-distributivity}).\end{proof}

\begin{proof}[Proof of (\ref{mainlifting-wcnotsh-suslinpreservation})]

Suppose $A \subset T$ is a maximal antichain with $A \in V[G \ast T \ast H_\alpha]$, and find a model $M$ such that $A \in M[G \ast T \ast H_\alpha]$. Using the lift $j:M[G \ast T \ast H_\alpha] \to N[\tilde G \ast \tilde T \ast \tilde H_\alpha]$, we will show that $j(A)$ is bounded in $\tilde T$. It will then follow by elementarity that $A$ is bounded in $T$.

Now let $\dot A$ be a $\P_\alpha$-name for $A$ in $V[G][ T]$ such that $\dot A \in M[G][T]$. Observe that by the elementarity of $j$, $j(A)\rest\lambda=A$; in particular, $A$ is a member of $N[\tilde G\ast \tilde T\ast \tilde H_\alpha]$. Recall that $t^*$ is the minimal extension of $T$ by the branch $B$. We will show that every node of $t^*$ at level $\lambda$ extends some element of $A$. This shows that $A$ is in fact a maximal antichain in $\tilde T$, and hence $j(A)=A$, since $j(A)\supseteq A$ is also a maximal antichain in $\tilde T$.

We recall that $t^*:=T \cup \{ s^\frown (B\setminus\gamma):s\in T\wedge\gamma<\lambda \}.$
By definition of $t^*$, to show that every element of $t^*$ on level $\lambda$ extends some element of $A$, it suffices to show that the following holds:

\begin{enumerate}
\item[$(*)$] for each $s\in T$ and $\gamma<\lambda$, there is some $\eta>\gamma$ so that $s^\frown (B\rest[\gamma,\eta))$ extends an element of $A$.
\end{enumerate} 

By upwards absoluteness, it suffices to verify that  $(*)$ holds in $V[G\ast T\ast (B\times H_\alpha)]$; this will use a density argument in $V[G \ast T]$ and the genericity of the branch $B$. We will now work over the model $V[G\ast T]$ to argue that for each $s\in T$ and $\gamma<\lambda$, the set
$$
\{ (u,r)\in\T\times\P_\alpha:(\exists\eta>\gamma)\; (u,r)\Vdash s^\frown(\dot{B}\rest[\gamma,\eta))\text{ extends a node in }\dot{A}\}
$$
denoted $D_{s,\gamma}$, is dense in $\T\times\P_\alpha$.

So fix some condition $(u,r)$ in $\T\times\P_\alpha$ as well as a node $s\in T$ and an ordinal $\gamma<\lambda$. Extend $u$ if necessary so that lh$(u)\geq\gamma$, noting that $u\Vdash_{\T}\dot{B}\rest\gamma=u\rest\gamma$. Set $\bar{u}:=u\rest[\gamma,\text{lh}(u))$ (it is possible that this is the empty sequence). Since $u\in\T$, $u$ is an element of $T$. By the homogeneity of $T$, $\bar{u}\in T$. As $s\in T$, we have by another use of homogeneity that $s^\frown\bar{u}\in T$. Recalling that $\dot{A}$ is forced by $\P_\alpha$ to be a maximal antichain in $T$, we may extend $r$ to a condition $r^*$ in $\P_\alpha$ and find a (possibly trivial) extension $s^\frown\bar{u}^\frown\bar{u}^*$ of $s^\frown\bar{u}$ in $T$ so that $r^*\Vdash s^\frown\bar{u}^\frown\bar{u}^*$ extends an element of $\dot{A}$. Now set $u^*:=u^\frown \bar{u}^*=(u\rest\gamma)^\frown\bar{u}^\frown\bar{u}^*$, and let $\eta$ be the length of $u^*$. We observe that $u^*\Vdash\dot{B}\rest[\gamma,\eta)=\bar{u}^\frown\bar{u}^*$. From this it now follows that $(u^*,r^*)$ forces that $s^\frown(\dot{B}\rest[\gamma,\eta))$ extends an element of $\dot{A}$, which completes the proof that $D_{s,\gamma}$ is dense. This in turn shows that $(*)$ holds in $V[G\ast T\ast(B\times H_\alpha)]$.

In summary, since $(*)$ holds, we conclude that every node in $t^*$ on level $\lambda$ extends an element of $A$. As shown earlier, we now see that $j(A)=A$, and hence $j(A)$ is bounded in $\tilde T$. By the elementarity of $j$, $A$ is bounded in $T$, completing the proof that $T$ is Suslin after forcing with $\P_\alpha$.\end{proof}

This completes the proof of Lemma~\ref{mainlifting-wcnotsh}.\end{proof}

\begin{lemma} $V[\M \ast \S] \models $``$\P_{\lambda^+}$ is $\lambda$-distributive.\end{lemma}

\begin{proof} As in Section~\ref{sec-tp-csr-ap}, this uses the $\lambda^+$-cc of the full iteration $\P$ together with the lifting lemma. The difference is that we use the definition of $\ell$ to specifically use a transitive $N$ and an elementary $j:M\to N$ with critical point $\lambda$ so that $j,M\in N$ and so that $j(\ell)(\lambda)=\dot{\bb{K}}_\alpha$.\end{proof}

Similarly, we have:

\begin{lemma} $V[\M \ast \S \ast \P] \models$ ``$T$ is a $\lambda$-Suslin tree'', i.e. $V[\M \ast \S \ast \P] \models \neg \SH(\lambda)$.\end{lemma}

\begin{lemma} $V[\M \ast \S \ast \P] \models \CSR(\lambda)$.\end{lemma}

Finally, we note that we have the property which was our reason for using Mitchell forcing to begin with:

\begin{lemma}\label{notsh-notap} $V[\M \ast \S \ast \P] \models \neg \AP(\kappa^{++})$.\end{lemma}

\begin{proof}  This is very similar to the analogous lemma for Theorem~\ref{tp-csr-notap}. We find the relevant $\vec a$ and $C$ in $M[G][T][H_\alpha]$ for some $\alpha<\lambda^+$ and lift a weakly compact embedding to $j:M[G][T][H_\alpha] \to N[\tilde G][\tilde T][\tilde H_\alpha]$. We can show that for the relevant tree $U$, we have a cofinal branch $e \in N[\tilde G][\tilde T][\tilde H_\alpha]$. Then we use the $j(\lambda)$-distributivity of $j(\S\ast\dot{\P}_\alpha)$ to show that $e \in N[\tilde G]=N[G][T][B][H_\alpha][I][G']$. Since $\N_{\lambda^*}$ is a projection of a product of a $\kappa^+$-closed poset with a square-$\kappa^+$-cc poset and $N[G][T][B][H_\alpha][I] \models$``$2^\kappa > \lambda$'', we see that $e \in N[G][T][H_\alpha][B][I]$. Then because $N[G][T][B][H_\alpha][I] \models $``$\lambda$ is regular'', we have a contradiction.\end{proof}

We obviously cannot make an argument for $\TP(\kappa^{++})$ work since the model has $\kappa^{++}$-Suslin trees. However, since the argument for $\neg \AP(\kappa^{++})$ deals with cofinal branches of trees, it is worth pointing out where exactly the argument for $\TP(\kappa^{++})$ will break down. Even though we could take a $\kappa^{++}$ tree $T \in M[G][T][H]$ and argue that it has a cofinal branch $b \in N[\tilde G][\tilde T][\tilde H]$, the forcing that adjoins $B$ makes it impossible to argue that $b \in N[G][T][H]$. By contrast, the argument for $\neg \AP(\kappa^{++})$ does not require us to show that the cofinal branch of interest is in $N[G][T][H]$.

\section{Results from Mahlo Cardinals}\label{sec-mahlo-results}

\subsection{Obtaining $\SR(\kappa^{++}) \wedge \wTP(\kappa^{++}) \wedge \AP(\kappa^{++})$}\label{stat-refl-sec}

In this section we will prove Theorem~\ref{wtp-sr-ap} using methods somewhat similar to our proof of Theorem~\ref{tp-csr-ap}. The reader should observe that while many of the steps are similar, the arrangement has non-trivial differences due to the nature of the embeddings we use for Mahlo cardinals.

We assume that $\kappa$ is regular with $\kappa^{<\kappa}=\kappa$, $\lambda$ is Mahlo, and $\kappa<\lambda$. We denote the poset $\M(\kappa,\lambda)$ from Definition~\ref{Def:Mitchell} as $\M$. Once more, $V[\M] \models \text{``}\kappa^{++} = \lambda$''. Again, we work in $V[\M]$ to define a standard club-adding iteration $\P_{\lambda^+}:=\seq{\P_\alpha,\dot \Q_\alpha}{\alpha < {\lambda^+}}$. As in the construction for Theorem~\ref{tp-csr-ap}, we define both the iteration and a sequence $\seq{\dot S_\alpha}{\alpha<{\lambda^+}}$ so that for each $\alpha$, there are $\beta$ and $\gamma$ so that $\dot S_\alpha$ will be the $\gamma^\text{th}$ $\P_\beta$-name for a subset of $\lambda \cap \cof(\le\!\kappa)$ that has no reflection points of cofinality $\kappa^+$. Our iteration will be defined so that for all $\alpha<{\lambda^+}$, $\P_{\alpha+1} = \P_\alpha \ast \dot{\Q}_\alpha = \P_\alpha \ast \textup{CU}(\lambda \setminus \dot{S}_\alpha)$.

\begin{proposition} $\P_{\lambda^+}$ has the $\lambda^+$-chain condition. \end{proposition}

Work in $V$ and let $N$ be a transitive set modeling a rich fragment of $\ZFC$ which is closed under $\lambda^+$-sequences in $V$. Here $N$ will be $H_\Theta$ for $\Theta\ge (2^{\lambda^+})^+$. (Hence $\M(\kappa,\lambda) \ast\P_{\lambda^+}$ is in $N$.) Recall the concept of rich submodels and the associated notation that was introduced in Section~\ref{mahlolaverfunction}.

\begin{lemma}\label{mainlifting2} 
For all $\alpha < {\lambda^+}$ and every rich $M$ with $\alpha\in M$, the following hold:

\begin{enumerate}

\item\label{mainlifting2-lift} For any $\bar \M\ast\bar{\P}_\alpha$-generic $\bar G\ast\bar H$ over $N$ there is an $\M\ast\P_\alpha$-generic $G*H$ over $N$ such that in $N[G*H]$ we can lift the elementary embedding $j:\bar M\then N$ to $j:\bar M[\bar G][\bar H]\then N[G][H]$.

\item\label{mainlifting2-distributivity} $\Vdash^N_\M \text{``}\P_\alpha$ is $\lambda$-distributive''.
\item\label{mainlifting2-stationarity} For any $\bar \M\ast\bar{\P}_\alpha$-generic\footnote{Note that $\bar{\P}_\alpha$ is an iteration of length $\bar{\alpha}$.} $\bar G\ast\bar H$ over $N$, $\bar {S}_\alpha$ is non-stationary in $N[\bar G][\bar H]$.
\end{enumerate}
\end{lemma}

\begin{proof} 

The proof is by induction on $\alpha<\lambda^+$. Let us assume that the lemma holds for all $\beta<\alpha$, and we will prove that it holds for $\alpha$. In particular, we show that (\ref{mainlifting2-lift}) holds at $\alpha$ using that (\ref{mainlifting2-stationarity}) holds for all $\beta<\alpha$, (\ref{mainlifting2-distributivity}) is an easy consequence of the proof of (\ref{mainlifting2-lift}), and finally, (\ref{mainlifting2-stationarity}) follows from (\ref{mainlifting2-lift}) and (\ref{mainlifting2-distributivity}) for $\alpha$. 

Let $M$ be a rich model with $\alpha\in M$.

\begin{proof}[Proof of (\ref{mainlifting2-lift})] First note that $\bar \M=\M(\kappa,\bar\lambda)$ since $M$ is closed under $<\bar\lambda$-sequences and the conditions in $\M(\kappa,\bar\lambda)$ are bounded below $\bar{\lambda}$. Let $\bar G\ast\bar H$ be an $\bar\M\ast\bar{\P}_\alpha$-generic filter over $N$.

\begin{Claim}\label{mainlifting2-denseclosedset} 
$N[\bar G] \models \text{``}\bar \P_\alpha$ has a $\bar \lambda$-closed dense subset''.
\end{Claim}

\begin{proof} 
By (\ref{mainlifting2-stationarity}) for all $\beta \in M \cap \alpha$, working in $N[\bar G]$ we can find $\bar \P_\beta$-names $\dot D_{\bar \beta}$ for club subsets of $\bar \lambda$ such that $\Vdash^{N[\bar{G}]}_{\bar \P_\beta} \dot D_{\bar \beta} \cap \dot{\bar{S}}_\beta = \emptyset$. It is straightforward to argue that $\mathcal{D} := \set {p \in \bar \P_\alpha}{\forall \bar \beta<\bar \alpha, p \rest \bar \beta \Vdash^{N[\bar{G}]}_{\bar{\P}_\beta} \max p(\bar \beta) \in \dot D_{\bar \beta}}$ is dense in $\bar{\P}_\alpha$ and $\bar \lambda$-closed.
\end{proof}

Work in $N[\bar G]$. Since $\bar\P_\alpha$ has size $\bar\lambda$, it holds that $\M/\bar G  \simeq  \M/\bar G \times \mathcal D \simeq \M/\bar G \times \bar \P_\alpha$ by the previous claim and Lemma \ref{absorption}. Let $G'$ be an $\M/\bar G$-generic over $N[\bar G][\bar H]$; then $N[\bar G][\bar H][G']$ can be written as $N[G]$, for some  $\M$-generic $G$ over $N$.

Now, working in $N[G]$, we lift the elementary embedding $j:\bar M\then N$ to $j:\bar M[\bar G]\then N[G]$. We know that $j$ is the identity below $\bar \lambda$ and the conditions in $\M(\kappa,\bar\lambda)$ are bounded in $\bar\lambda$, and so $j[\bar G]=\bar G \sub G$, and we can lift by Fact \ref{weakliftinglemma}.

To lift the elementary embedding further we will find a strong master condition $q$ for $j$ and $\bar{\P}_\alpha$. Work in $N[G]$: Let $\seq{C_{\bar\beta}}{\bar \beta<\bar{\alpha}}$ be the sequence of club sets in $\bar{\lambda}$ added by $\bar{H}$, noting that this sequence is in $N[G]$ since $\bar{H}$ is. Let us define $q$ such that $\dom{q}=j[\bar{\alpha}]$ and for $\beta\in j[\bar{\alpha}]$, $q(\beta)=C_{\bar\beta} \cup\{\bar{\lambda}\}$. Since we ensured $\bar{H} \in N[G]$, $q$ is an element of $N[G]$.

It suffices to verify that $q$ is a condition in $\P_\alpha$ since it is by definition below every condition in $j[\bar{H}]$. This follows as $\bar{\lambda}$ is an ordinal of cofinality $\kappa^+$ below $\lambda$ in $N[G]$ and therefore for each $\bar\beta<\bar\alpha$, $C_{\bar\beta}\cup\{\bar{\lambda}\}$ is a closed bounded subset of $\lambda\setminus{S_{\beta}}$.

Let $H$ be a $\P_\alpha$-generic over $N[G]$ which contains $q$; then $j[\bar H]\sub H$ and we can lift the elementary embedding $j:\bar M[\bar G]\then N[G]$ to $j:\bar M[\bar G][\bar H]\then N[G][H]$ by Fact \ref{weakliftinglemma}.
\end{proof}

\begin{proof}[Proof of  (\ref{mainlifting2-distributivity})]
By elementarity, it is enough to show that ${\bar\M}$ over $\bar M$ forces that $\bar {\P}_\alpha$ is $\bar \lambda$-distributive. Assume for a contradiction that there is a condition $\bar m$ which forces that $\bar{\P}_\alpha$ is not $\bar\lambda$-distributive. Let $\bar G$ be an $\bar\M$-generic over $N$ which contains $\bar m$. It means that in $\bar M[\bar G]$ it holds that $\bar{\P}_\alpha$ is not $\bar\lambda$-distributive, therefore there is a condition $\bar p\in \bar{\P}_\alpha$ and a $\bar{\P}_\alpha$-name $\dot{x}$ such that $\bar p \Vdash \text{``}\dot{x}$  is a new sequence of ordinals of length less than $\bar\lambda$''. Let $\bar H$ be a $\bar{\P}_\alpha$-generic over $N[\bar G]$ which contains $\bar{p}$: then $x = \dot{x}^{\bar{H}}$ is a sequence of ordinals of length less than $\bar\lambda$ which is in $\bar M[\bar G][\bar H]$ but not in $\bar M[\bar G]$.

However, by the proof of (\ref{mainlifting2-lift})  we can lift the embedding $j$ from $\bar M[\bar G]$ to $N[G]$ for some $G$, and moreover there is in $N[G]$ a condition $q\in j(\bar{\P}_\alpha)=\P_\alpha$ below $j[\bar H]$. It follows that $\bar M[\bar G]$ and $\bar M[\bar G][\bar H]$ must have the same sequences of ordinals of length less than $\bar\lambda$ by Fact \ref{strongliftinglemma}, which is a contradiction.
\end{proof}

\begin{proof}[Proof of  (\ref{mainlifting2-stationarity})]
Let $\bar G\ast\bar H$ be an $\bar\M\ast\bar \P_\alpha$-generic over $N$, and let us lift the elementary embedding $j:\bar M\then N$ as in the proof of (\ref{mainlifting2-lift}) to $j:\bar M[\bar G][\bar H]\then N[G][H]=N[\bar G][\bar H][G'][H]$ where $G*H$ is an $\M\ast \P_\alpha$-generic over $N$ and $G'$ is an $\M/\bar G$-generic over $N[\bar G][\bar H]$.

First note that $j(\bar S_\alpha)=S_\alpha$ and $j(\bar S_\alpha)\cap\bar\lambda=\bar S_\alpha$ since $j$ is the identity below $\bar\lambda$. Therefore $\bar S_\alpha$ must be non-stationary in $N[\bar G][\bar H][G'][H]$ since $S_\alpha$ is non-reflecting and $\bar\lambda$ is an ordinal of cofinality $\kappa^+$ in $N[\bar G][\bar H][G'][H]$. Therefore it is enough to show that $\M/\bar G*\P_\alpha$ (for which $G' *H$ is generic) does not destroy stationary subsets of $\bar\lambda\cap \cof (\le\!\kappa)$ over $N[\bar G][\bar H]$, and hence $\bar S_\alpha$ is non-stationary in $N[\bar G][\bar H]$ as well.

By (\ref{mainlifting2-distributivity}), $\P_\alpha$ is $\lambda$-distributive over $N[\bar G][\bar H][G']$. Since $\lambda>\bar \lambda$ and $\bar{S}_\alpha\sub\bar\lambda$, $\bar S_\alpha$ is non-stationary in $N[\bar G][\bar H][G']$. By Lemma \ref{preservation}, $\M/\bar G$ does not destroy stationary subsets of $\bar\lambda\cap \cof(\le\!\kappa)$ over $N[\bar G][\bar H]$. Therefore $\bar S_\alpha$ is already non-stationary in $N[\bar G][\bar H]$.
\end{proof}

This finishes the whole proof of Lemma \ref{mainlifting2}.
\end{proof}

\begin{lemma} 
$V[\M] \models$ ``$\P_{\lambda^+}$ is $\lambda$-distributive''.
\end{lemma}

\begin{proof} As in the proof of Lemma~\ref{wcdist} we could choose $N$ to contain a supposed counterexample $\dot f$, so this follows from Lemma~\ref{mainlifting2}(\ref{mainlifting2-distributivity}).
\end{proof}

\begin{proposition} $V[\M \ast \P_{\lambda^+}] \models \SR(\lambda)$.\end{proposition}

\begin{proof} As $\P_{\lambda^+}$ is $\lambda$-distributive and we have preservation of $\lambda$, this follows by construction and from Fact~\ref{nicenamestrick}.\end{proof}

\begin{lemma} 
$V[\M \ast \P_{\lambda^+}] \models \AP(\lambda)$.
\end{lemma}

\begin{proof} 
This follows from $V[\M] \models \AP(\lambda)$ and the fact that $\P_{\lambda^+}$ is $\lambda$-distributive.
\end{proof}

\begin{lemma}
\label{Mahlo-wtp} $V[\M \ast \P_{\lambda^+}] \models \wTP(\lambda)$.
\end{lemma}

\begin{proof} 
Assume for contradiction that $\M*\P_{\lambda^+}$ adds a special $\lambda$-Aronszajn tree. By Fact \ref{nicenamestrick} there is $\alpha<\lambda^{+}$ such that the tree and a specializing function for it are already added  by $\M*\P_\alpha$. Let us fix this $\alpha$. Let $\dot{T}$ be an $\M*\P_\alpha$-name for a special $\lambda$-Aronszajn tree, $\dot{f}$ a name for a specialization function for $\dot{T}$ and $(m,p) \in \M* \P_\alpha$ a condition which forces these facts. Following the notation in Definition \ref{def:rich}, let $\Theta$ be sufficiently large with $N=H_\Theta$, and let $M$ be a rich model in $N$ such that $\dot{T}$, $\dot{f}$, $(m,p)$ and $\alpha$ are in $M$.

Let $\bar G\ast\bar H$ be a $\bar \M\ast\bar{\P}_\alpha$-generic over $N$ which contains $(\bar m,\bar p)$. As in the proof of Lemma \ref{mainlifting2}(\ref{mainlifting2-lift}), lift the embedding $j$ from $\bar{M}[\bar{G}][\bar{H}]$ to $N[G][H]=N[\bar G][\bar H][G'][H]$, where $G*H$ is $\M\ast \P_\alpha$-generic over $N$ and $G'$ is $\M/\bar G$-generic over $N[\bar G][\bar H]$. In $N[G][H]$, $T = \dot{T}^{G*H}$ is a special $\lambda$-Aronszajn tree with the specialization function $f = \dot{f}^{G*H}$ since $j((\bar m,\bar p))=(m,p)\in G\ast H$. As we assumed that $\dot{T}$ and $\dot{f}$ are in $M$, there are $\bar{T}$ and $\bar{f}$ in $\bar{N}[\bar{G}][\bar{H}]$ such that $j(\bar{T})=T$ and $j(\bar{f})=f$. Moreover, as $j$ is the identity below $\bar{\lambda}$ we have $T\rest\bar{\lambda}=\bar{T}$ and $f\rest\bar{\lambda}=\bar{f}$. Note that $\bar{T}$ has a cofinal branch $b$ in $N[G][H]$ since $T\rest\bar{\lambda}=\bar{T}$. We will show that this branch is already in $N[\bar{G}][\bar{H}]$, and this will be a contradiction as $\bar{f}\rest b$ (modulo some bijection between $\bar{\lambda}$ and $b$) collapses $\bar{\lambda}$ to $\kappa^+$.  This will contradict the fact that $\bar{\M}*\bar{\P}_{\alpha}$ preserves $\bar{\lambda}$.

First, note that $b$ cannot be added by $\P_{\alpha}$ as $\P_{\alpha}$ is $\lambda$-distributive ($\lambda>\bar\lambda$) over $N[G]=N[\bar G][\bar H][G']$ and $b$ has length $\bar{\lambda}$.  To finish the proof it suffices to show that $\M/\bar{G}$ over $N[\bar{G}][\bar{H}]$ cannot add $b$, and this follows from Lemma \ref{preservation} since $\bar\P_\alpha$ has a $\bar\lambda$-closed dense subset. 
\end{proof}

\subsection{A Guessing Argument for $\SR(\kappa^{++}) \wedge \neg\AP(\kappa^{++})$ from a Mahlo Cardinal}

Now we are in a position to perform a new proof of $\SR(\kappa^{++}) \wedge \neg \AP(\kappa^{++})$ from a Mahlo cardinal. We will explain it in terms of the modifications that must be made to the proof of Theorem~\ref{wtp-sr-ap}.

Let $\kappa<\lambda$ again be a regular cardinal such that $\kappa^{<\kappa}=\kappa$. We will assume both $V \models \text{``}\lambda$ is Mahlo'' and that $V \models \diamondsuit_\lambda(\textup{Reg})$, which follows from the consistency of a Mahlo cardinal---for example, in $L$.  Let $\ell:\lambda \to V_\lambda$ witness statement (2) from Proposition~\ref{mahlolaverfunction}. Let $\M=\M_\ell(\kappa,\lambda)$ be the guessing poset from Definition~\ref{def:myMitchell}, defined in terms of $\ell$. Let $\P_{\lambda^+}$ and $N$ be as in the construction for Theorem~\ref{wtp-sr-ap}.

\begin{lemma}\label{mainlifting2-notap} 
For all $\alpha < {\lambda^+}$ and every rich $M$ with $\alpha\in M$ and $\ell(\bar \lambda_M)=\pi_M(\P_\alpha)$, the following hold:

\begin{enumerate}

\item\label{mainlifting2-notap-lift} For any $\bar \M\ast\bar{\P}_\alpha$-generic $\bar G\ast\bar H$ over $N$ there is an $\M\ast\P_\alpha$-generic $G*H$ over $N$ such that in $N[G*H]$ we can lift the elementary embedding $j:\bar M\then N$ to $j:\bar M[\bar G][\bar H]\then N[G][H]$.

\item\label{mainlifting2-notap-distributivity} $\Vdash^N_\M \text{``}\P_\alpha$ is $\lambda$-distributive''.
\item\label{mainlifting2-notap-stationarity} For any $\bar \M\ast\bar{\P}_\alpha$-generic $\bar G\ast\bar H$ over $N$, $\bar {S}_\alpha$ is non-stationary in $N[\bar G][\bar H]$.
\end{enumerate}
\end{lemma}

\begin{proof} The proof is very similar to that of Lemma~\ref{mainlifting2}; the only difference is in the proof that (\ref{mainlifting2-lift}) holds at $\alpha$ assuming that (\ref{mainlifting2-stationarity}) holds for all $\beta<\alpha$. Again, we let $\bar G\ast\bar H$ be an $\bar\M\ast\bar{\P}_\alpha$-generic over $N$.

Once we have established that $N[\bar G] \models $``$\bar \P_\alpha$ has a $\bar \lambda$-closed dense subset'', we can use the guessing property of $\ell$ to conclude that we have the forcing equivalence $\M/\bar G \simeq (\bar{\P}_\alpha \times\Add(\kappa,\lambda^*)) \ast  \dot{\N}_{\lambda^*}$ where $\lambda^*$ is the next inaccessible after $\bar \lambda$. Let $\bar H \times I$ be $\bar \P_\alpha \times \Add(\kappa,\lambda^*)$-generic over $N[\bar G]$ and let $G'$ be $\N_{\lambda^*}$-generic over $N[\bar G][\bar H \times I]$. Therefore, it follows that we can express $N[G]=N[\bar G][\bar H \times I][G']$ where $G$ is $\M$-generic over $N$. Adjoining $I \ast G'$ similarly preserves stationary sets and does not add branches to the relevant trees. The rest of the proof proceeds as before.\end{proof}

As before we find that $V[\M] \models$``$\P_{\lambda^+}$ is $\lambda$-distributive'' and $V[\M \ast \P_{\lambda^+}] \models \SR(\kappa^{++})$.

\begin{lemma} $V[\M \ast \P_{\lambda^+}] \models \neg \AP(\kappa^{++})$.\end{lemma}

\begin{proof} This argument will be almost the same as the one from Lemma~\ref{wc-notap}. We suppose for contradiction that $V[\M \ast \P_{\lambda^+}] \models \AP(\kappa^{++})$, which means that this will be witnessed in $N[\M \ast \P_\alpha]$ for some $\alpha<\lambda^+$ with respect to an enumeration $\vec{a}=\seq{a_i}{i<\lambda}$ of all bounded subsets of $\lambda$ in that model. Find $(m,p) \in \M \ast \P_\alpha$ forcing this, let $M$ be a rich model containing these parameters, and let $\bar G \ast \bar H$ be an $\bar \M \ast \bar \P_\alpha$-generic containing $(\bar m,\bar p)$. We will use a lift $j:\bar M[\bar G][\bar H] \to N[G][H]$. Let $C$ be a club of points that are approachable with respect to $\vec a$, and note that $\bar \lambda \in C$. Let $e$ be a short club witnessing the approachability of the point $\bar \lambda$, and consider the tree $U:=({\bar \lambda}^{< \kappa^+})^{N[\bar G][\bar H]}$. We see that $U$ has a cofinal branch of length $\kappa^+$ given by $e$ which is in $N[\bar G][\bar H][I]$ because of the properties of the extension by $H'$, and this in turn leads us to a contradiction because $N[\bar G][\bar H][I] \models \text{``}\bar \lambda = \kappa^{++}$''.\end{proof}

\subsection{Obtaining $\SR(\kappa^{++}) \wedge \neg \SH(\kappa^{++}) \wedge \neg \AP(\kappa^{++})$}

The construction for Theorem~\ref{sr-notsh} combines ideas from Theorems~\ref{gk} and \ref{csr-notsh}. Assume again that $V \models$``$\lambda$ is Mahlo and $\diamondsuit_\lambda(\textup{Reg})$ holds'', and let $\ell:\lambda \to V_\lambda$ witness the alternate formulation of this diamond from Proposition~\ref{mahlolaverfunction}. Let $\kappa<\lambda$ be regular with $\kappa^{<\kappa} = \kappa$, let $\M = \M_\ell(\kappa,\lambda)$, let $\S$ be Kunen's forcing for adding a $\kappa^{++}$-Suslin tree as defined in $V[\M]$, and let $\P$ be the club-adding iteration from Theorems~\ref{wtp-sr-ap} and \ref{gk} defined in terms of an enumeration $\seq{\dot{S}_\alpha}{\alpha<\lambda^+}$. Our intended model will then be $V[\M \ast \S \ast \P]$. To this end, we let $N=H_\Theta$ for a large enough $\Theta$ to prove that the extension has the needed properties.

The main task will be to establish a lifting lemma. We again use the notation in which, working in $V[\M]$, we define $\bb{K}_\alpha = \S \ast (\T \x \P_\alpha)$.

\begin{lemma}\label{mainlifting-mahlonotsh} For all $\alpha < \lambda^+$ and every rich $M \prec N$ with $\alpha \in M$ and $\ell(\bar \lambda_M) = \pi_M(\bb{K}_\alpha)$, the following hold:

\begin{enumerate}
\item\label{mainlifting-mahlonotsh-embedding} For any $\bar \M \ast \bar \S \ast \bar \P_\alpha$-generic $\bar G \ast \bar T \ast \bar H$ over $N$ there is an $\M \ast \S \ast \P_\alpha$-generic $G \ast T \ast H$ such that in $N[G \ast T \ast H]$ we can lift the elementary embedding $j:\bar M \to N$ to $j:\bar M[\bar G][\bar T][\bar H] \to N[G][T][H]$.
\item\label{mainlifting-mahlonotsh-distributivity} $\Vdash_{\M \ast \S}^N$``$\P_\alpha$ is $\lambda$-distributive''.
\item\label{mainlifting-mahlonotsh-suslinpreservation} $\Vdash_{\M \ast \S}^N$``$\P_\alpha$ forces that the tree added by $\S$ remains Suslin''.
\item\label{mainlifting-mahlonotsh-stationarity} For any $\bar \M \ast \bar \S \ast \bar \P_\alpha$-generic $\bar G \ast \bar T \ast \bar H$ over $N$, $\bar S_\alpha$ is non-stationary in $N[\bar G][\bar T][\bar H]$.
\end{enumerate}
\end{lemma}

\begin{proof} We will prove that (\ref{mainlifting-mahlonotsh-embedding}) holds for $\alpha$ assuming (\ref{mainlifting-mahlonotsh-stationarity}) for $\beta<\alpha$. Then (\ref{mainlifting-mahlonotsh-distributivity}) and (\ref{mainlifting-mahlonotsh-suslinpreservation}) will both be consequences of (\ref{mainlifting-mahlonotsh-embedding}) for $\alpha$. Finally, we use (\ref{mainlifting-mahlonotsh-embedding}) and (\ref{mainlifting-mahlonotsh-distributivity}) to prove (\ref{mainlifting-mahlonotsh-stationarity}) for $\alpha$. 

\begin{proof}[Proof of (\ref{mainlifting-mahlonotsh-embedding})] We have $\bar \M = \M_\ell(\kappa,\bar \lambda)$ by the closure properties of $M$. Let $\bar G \ast \bar T \ast \bar H$ be $\bar \M \ast \bar \S \ast \bar \P_\alpha$-generic over $N$.

We start by working in $N[\bar G]$, so that we can lift to $j:\bar M[\bar G] \to N[G]$.

\begin{Claim} $N[\bar G] \models$``$\bar{\mathbb{K}}_\alpha$ has a $\bar \lambda$-closed dense subset.\end{Claim}

\begin{proof}[Proof of Claim] By elementarity, $\bar \S \ast \bar \T$ has a $\bar \lambda$-closed dense subset $D(\bar \S \ast \bar \T)$. Using (\ref{mainlifting-mahlonotsh-stationarity}) for $\beta \in M \cap \alpha$, we find $\bar \S \ast (\bar \T \times \bar \P_\beta)$-names $\dot D_{\bar \beta}$ for clubs in $\bar \lambda$ avoiding $S_{\bar \beta}$. Then we can see that

$$\{(s,t,p): (s,t) \in D(\bar \S \ast \bar \T) \wedge \forall \beta \in M \cap \alpha, (s,t,p \rest \bar \beta) \Vdash_{\bar \S \ast (\bar \T \times \bar \P_\beta)}^{N[\bar G]} \max p(\bar \beta) \in \dot D_{\bar \beta} \}$$

\noindent is $\bar \lambda$-closed and dense in $\bar{\bb{K}}_\alpha$.\end{proof}

We again abuse notation, letting $\bar{\bb{K}}_\alpha$ denote this dense subset. Then in $N[\bar G]$ we have the equivalence $\M/\bar G \simeq (\bar{\bb{K}}_\alpha \times \Add(\kappa,\lambda^*)) \ast \dot {\N}_{\lambda^*}$, where $\lambda^*$ is the least inaccessible above $\bar \lambda$. Next let $\bar B$ be a generic branch through $\bar T$ and $I$ an $\Add(\kappa,\lambda^*)$-generic over $N[\bar G \ast (\bar T \ast (\bar B \times \bar H)]$ so that $G:=((\bar T \ast (\bar B \times \bar H)) \times I) \ast G'$ is $\M$-generic over $N$. Then we can lift the embedding to $j:\bar M[\bar G] \to N[G]$.

To lift $j$ to have domain $\bar M[\bar G \ast \bar T]$, we choose $t^*$ to be the minimal extension of $\bar T$ by $\bar B$ (as in the proof of Lemma~\ref{mainlifting-wcnotsh}(\ref{mainlifting-wcnotsh-embedding})) and let $T$ be $j(\S)$-generic over $N[G]$ such that $t^* \in T$. Then we lift to $j:\bar M[\bar G \ast \bar T] \to N[G \ast T]$.

Lifting $j$ to have domain $\bar M[\bar G \ast \bar T \ast \bar H]$ proceeds as in the proof of Lemma~\ref{mainlifting2}(\ref{mainlifting2-lift}): We find a strong master condition $q$ for $j$ and $\bar \P_\alpha$, and we are able to verify that it is in fact a condition from the fact that $N[G \ast T] \models$``$\cf(\bar \lambda) = \kappa^+$''.\end{proof}

\begin{proof}[Proof of (\ref{mainlifting-mahlonotsh-distributivity})] This is exactly as in the proof of Lemma~\ref{mainlifting2}(\ref{mainlifting2-distributivity}).\end{proof}

\begin{proof}[Proof of (\ref{mainlifting-mahlonotsh-suslinpreservation})] We argue as in the proof of Lemma~\ref{mainlifting-wcnotsh}(\ref{mainlifting-wcnotsh-suslinpreservation}), making the necessary adjustments to suit our present setting.

Suppose for a contradiction that there is some $(s,p) \in \S \ast \P_\alpha$ forcing over $N[\bar G]$ that the generic tree added by $\S$ is not Suslin as witnessed by an unbounded antichain with an $\S \ast \P_\alpha$-name $\dot A$. Find a rich model $M \prec N$ containing $(s,p)$ and $\dot A$. Let $\bar T \ast \bar H$ be an $\bar \S \ast \bar \P_\alpha$-generic containing $(\bar s,\bar p)$ and use (\ref{mainlifting-mahlonotsh-embedding}) to obtain a lift $j:\bar M[\bar G \ast \bar T \ast \bar H] \to N[G \ast T \ast H]$ where $G \ast T \ast H$ is the relevant $\M \ast \S \ast \P_\alpha$-generic. In particular, we are using that $T$ contains $t^\ast$, the minimal extension of $\bar T$ by $\bar B$ used in the proof. Observe that, as in the proof of Lemma~\ref{mainlifting-wcnotsh}(\ref{mainlifting-wcnotsh-suslinpreservation}), every node of $t^*$ at level $\bar \lambda$ extends an element of $\bar A$ (essentially the same argument works). Hence $\bar A$ is a maximal antichain in $T$ and hence $\bar A = A$, meaning that $A$ is bounded in $T$. This is a contradiction of our assumptions about the name $\dot A$ and the information forced by $(s,p)$.\end{proof}

\begin{proof}[Proof of (\ref{mainlifting-mahlonotsh-stationarity})] This is as in the proof of Lemma~\ref{mainlifting2}(\ref{mainlifting2-notap-stationarity}), but with an extra application of Fact~\ref{F:St_cc} to explain that adjoining the $\Add(\kappa,\lambda^*)$-generic $I$ does not destroy stationary subsets of $\bar \lambda \cap \cof(\le \kappa)$.
\end{proof}

This completes the proof of Lemma~\ref{mainlifting-mahlonotsh}.\end{proof}

We find that $V[\M \ast \S \ast \P] \models \SR(\kappa^{++}), \neg\AP(\kappa^{++})$ as in previous arguments.

\section{Obtaining Large $2^\kappa$}\label{sec-large-cont}

Now we will obtain our previous results with large $2^\kappa$. This will be done by forcing $\Add(\kappa,\mu)$ for arbitrarily large $\mu$ of cofinality $>\kappa$. Although this might seem easy, we are faced with two problems: the preservation of the stationary reflection properties, and the preservation of the properties pertaining to trees. The first problem is dealt with using some preservation theorems of the third author. The second requires us to use the embeddings constructed in the previous sections and work around the fact that extra copies of Cohen forcing cannot be absorbed into the quotients of Mitchell forcing (neither using the actual absorption lemma nor using the guessing trick with $\M_\ell(\kappa,\lambda)$).

In several instances we will make use of the fact that we can reinterpret certain iterations as products:

\begin{observation}\label{product-trick} If $\P \ast \Q$ is a two-step iteration where $\P$ is $\delta$-distributive, conditions in $\Q$ are forced to have size $<\delta$, and $\P$ does not change the definition of $\Q$, then we have the forcing-equivalence $\P \ast \Q \simeq \P \times \Q$. In particular, if $\P$ is $\delta$-distributive, then $\P \ast \Add(\kappa,\beta) \simeq \P \times \Add(\kappa,\beta)$ for any $\kappa \le \delta$ and any $\beta$.\end{observation}

We also need a stronger branch-preserving lemma that allows us to deal with trees that have been added by $\kappa^+$-cc forcings:

\begin{fact}\label{F:ccc_Closed}\cite{UNGER:1}
Let $\kappa$ be regular, and let $\delta$ be an ordinal of cofinality at least $\kappa^+$. Assume that $\P$ and $\Q$ are forcing notions such that $\P$ is $\kappa^{+}$-cc and $\Q$ is $\kappa^{+}$-closed. If $T$ is a tree in $V[\P]$ of height $\delta$ and levels of size $<2^\kappa$, then forcing with $\Q$ over $V[\P]$ does not add cofinal branches to $T$.\end{fact}

We continue in the first subsection with general results about the preservation of approachability. Afterwards, we have subsections which address large $2^\kappa$  in the context of our weakly compact results and then in the context of our Mahlo results.

\subsection{Preservation of the Failure of Approachability}

For preserving the failure of approachability, we use a result from \cite{GK:a} that $\neg \AP(\kappa^{++})$ is preserved under all $\kappa$-centered forcings. Recall that $\Q$ is \emph{$\kappa$-centered} for a regular $\kappa$ if $\Q$ can be written as the union of a family $\set{\Q_\alpha \sub \Q}{\alpha <\kappa}$ such that for every $\alpha<\kappa$ and every $p,q \in \Q_\alpha$ there exists $r \in \Q_\alpha$ with $r \le p,q$.

\begin{fact}\cite{GK:a}\label{th:AP}
Assume $\neg \AP(\kappa^{++})$ holds and $\Q$ is $\kappa$-centered. Then the forcing $\Q$ forces $\neg \AP(\kappa^{++})$.
\end{fact}

\begin{lemma}\label{Cohen:centered}
Suppose $\kappa^{<\kappa} = \kappa$ and $\nu \le 2^\kappa$. Then $\Add(\kappa,\nu)$ is $\kappa$-centered.
\end{lemma}

\begin{proof}
 We give a sketch of proof for the benefit of the reader. By a standard $\ZFC$ argument, for every $\nu \le 2^\kappa$, there is a system $\set{M(\alpha,i)}{\alpha < \nu, i <2}$ of subsets of $\kappa$ which forms an independent system. Namely for each $\alpha<\nu$, $M(\alpha,0)$ and $M(\alpha,1)$ form a partition of $\kappa$ into two infinite sets and that for every function $f: x \to 2$ where $x \sub \nu$ has size $<\!\kappa$, the intersection of $\set{M(\alpha,f(\alpha))}{\alpha \in x}$ has size $\kappa$. Let us view $\Add(\kappa,\nu)$ (equivalently) as functions from subsets of $\nu$ of size $<\!\kappa$ to 2. For each $\alpha < \kappa$, define $F_\alpha: \nu \to 2$ by $F_\alpha(\beta) = i$, for the unique $i$ with $\alpha \in M(\beta,i)$. Let $P_\alpha = \set{p \in \Add(\kappa,\nu)}{p \mbox{ compatible with }F_\alpha}$. Each $P_\alpha$ is clearly centered. Let us check that $\bigcup_{\alpha<\kappa}P_\alpha = \Add(\kappa,\nu)$. For $p \in \Add(\kappa,\nu)$ with domain $J$ of size $<\!\kappa$, the intersection $X_p = \bigcap\set{M(\beta,p(\beta))}{\beta \in J}$ is non-empty. For any $\alpha \in X_p$, $p$ is compatible with $F_\alpha$: clearly, $F_\alpha(\beta) = p(\beta)$ for any $\beta\in J$. It follows that $p \in P_\alpha$, and we are done.
\end{proof}

\begin{proposition}\label{AP-trick} 
Assume $\kappa^{<\kappa}=\kappa$ and $\mu \ge \kappa^{++}$. If $\Add(\kappa,\mu)$ forces $\AP(\kappa^{++})$, then so does $\Add(\kappa,\kappa^{++})$.
\end{proposition}

\begin{proof} 
Assume that $\Add(\kappa,\mu)$ forces $\AP(\kappa^{++})$.
The approachability of $\kappa^{++}$ is witnessed by a sequence of bounded subsets of $\kappa^{++}$, $\vec{a}=\seq{a_\alpha}{\alpha<\kappa^{++}}$,  a closed unbounded set $C\sub\kappa^{++}$ and by a sequence $A=\seq{A_\alpha}{\alpha\in C\mbox{ and } A_\alpha\sub \alpha}$, where each $A_\alpha$ is cofinal in $\alpha$. Since all these objects have size $\kappa^{++}$ and $\Add(\kappa,\mu)$ is $\kappa^+$-cc, we can fix $\Add(\kappa,\mu)$-nice names $\dot{\vec{a}}$, $\dot{C}$, and $\dot{A}$ for $\vec{a}$, $C$ and $A$, respectively, each of size $\kappa^ {++}$.

Let $B\sub\mu$ be the set of  coordinates in the Cohen forcing which appear in the names $\dot{\vec{a}}$, $\dot{C}$ or $\dot{A}$; $B$ has size at most $\kappa^{++}$. Then the Cohen forcing with coordinates in $B$ forces $\AP(\kappa^{++})$, and hence also $\Add(\kappa,\kappa^{++})$ forces it because any bijection between $B$ and $\kappa^{++}$ determines an isomorphism between $\Add({\kappa},B)$ and $\Add({\kappa},\kappa^{++})$.
\end{proof}

\begin{remark}
The use of the preservation argument makes the following proofs shorter, but this is not necessary. Arguments based on a quotient analysis may be used for failure of $\AP(\kappa^{++})$ as well, with the proofs being analogous to those of Lemmas \ref{tp-large-cont} and \ref{wtp-large-cont} below.
\end{remark}

\subsection{The Weakly Compact Results with Large $2^\kappa$}

In this subsection we prove Theorem~\ref{wc-large-cont}, or rather, we prove Theorems~\ref{tp-csr-ap}, \ref{tp-csr-notap}, and \ref{csr-notsh} for arbitrarily large $\mu$. Assuming that $\cf(\mu)>\kappa$, we will even show that it is enough to force with $\Add(\kappa,\mu)$ over any of these models to obtain the desired result. Naturally, it is immediate that $2^\kappa = \mu$ in the extension by $\Add(\kappa,\mu)$.

For preserving club stationary reflection at $\kappa^{++}$ in each of these cases, we use a recent result of the third author (see \cite{HS:u}):

\begin{fact}\label{th:csr}
Suppose $\lambda = \kappa^{++}$, $\kappa^{<\kappa} = \kappa$ and $\CSR(\lambda)$ holds. Then $\Add(\kappa,\mu)$ preserves $\CSR(\lambda)$ for any $\mu$.
\end{fact}

Hence we have that for any of these models $W$, $W[\Add(\kappa,\mu)] \models \CSR(\lambda)$ as an immediate corollary. Furthermore, since a product of a $\nu$-Knaster forcing with a $\nu$-cc forcing is $\nu$-cc, one can observe that forcings with the $\nu$-Knaster property do not add $\nu$-sized antichains to $\nu$-Suslin trees. Therefore, if $W$ is the model from Theorem~\ref{csr-notsh}, then $W[\Add(\kappa,\mu)] \models \neg \SH(\kappa^{++})$.

\begin{lemma}\label{notap-large-cont} If $W$ is either the model from Theorem~\ref{tp-csr-notap} or Theorem~\ref{csr-notsh}, then $W[\Add(\kappa,\mu)] \models \neg \AP(\kappa^{++})$.\end{lemma}

\begin{proof} 
This is an immediate consequence of the Gitik and Krueger preservation result and Lemma \ref{Cohen:centered} and Proposition \ref{AP-trick}.
\end{proof}

This means that it suffices to verify that $\TP(\lambda)$ still holds after forcing with $\Add(\kappa,\mu)$ over the models of Theorems~\ref{tp-csr-ap} and \ref{tp-csr-notap}. This requires a more involved approach as we do not have any general preservation lemma for the tree property.

Next, we show that we can actually consider extensions by $\Add(\kappa,\lambda)$ instead of $\Add(\kappa,\mu)$.

\begin{proposition}\label{a-tree-trick} Working over a ground model $W$, if $\P_{\lambda^+}$ is a standard club-adding iteration and there is a $\lambda$-Aronszajn tree $T \in W[\P_{\lambda^+} \ast \Add(\kappa,\mu)]$, then $T \in W[\P_\alpha \ast \Add(\kappa,\lambda)]$ for some $\alpha<\lambda^+$.
\end{proposition}

\begin{proof} If $T$ is a $\lambda$-Aronszajn tree in $W[\P_{\lambda^+} \ast \Add(\kappa,\mu)]$ then since $\P_{\lambda^+} \ast \Add(\kappa,\mu)$ is $\lambda^+$-cc, $T$ has a $\P_{\lambda^+} \ast \Add(\kappa,\mu)$-nice name $\dot{T}$ which has size $\lambda$. Because the iteration is $\lambda^+$-cc, we can apply a mild generalization of Fact~\ref{nicenamestrick} to find $\alpha < \lambda^+$ such that $\dot T$ is a $\P_\alpha \ast \Add(\kappa,\mu)$-name. Moreover, there is some $A\sub\mu$ of cardinality $\lambda$ such that $\P_\alpha \ast \Add(\kappa,A)$ determines the support of conditions in $\dot{T}$. Any bijection between $A$ and $\lambda$ gives an isomorphism between $\Add(\kappa,A)$ and $\Add(\kappa,\lambda)$, so we can assume $\dot T$ is a $\P_\alpha \ast \Add(\kappa,\lambda)$-name.\end{proof}

\begin{lemma}\label{tp-large-cont} If $W$ is either the model from Theorem \ref{tp-csr-ap} or Theorem \ref{tp-csr-notap}, then $W[\Add(\kappa,\mu)] \models \TP(\lambda)$.\end{lemma}

\begin{proof} We perform the proof in two cases, the second by indicating the changes that need to be made to the first.

\emph{Case 1:} $W$ is the model of $\CSR(\kappa^{++}) \wedge \TP(\kappa^{++}) \wedge \AP(\kappa^{++})$ from Theorem \ref{tp-csr-ap}.

By Proposition~\ref{a-tree-trick} it is enough to show that for any $\alpha<\lambda^+$, $V[\M\ast\P_\alpha\ast\Add(\kappa,\lambda)]$ satisfies the desired conclusion. So fix $\alpha<\lambda^+$, and let $G\ast H\ast K$ be $\M\ast\P_\alpha\ast\Add(\kappa,\lambda)$-generic over $V$. Suppose $T$ is a $\lambda$-tree in $M[G ][ H ][ K]$ and let $j:M \to N$ witness weak compactness of $\lambda$. We will show that there is a forcing extension of $N[G ][ H ][ K]$ that contains a cofinal branch $b$ of $T$, and then that $b$ is actually already in $N[G ][ H ][ K]$, from which it will follow that $b \in V[G ][ H ][ K]$.

First we will find the forcing extension containing the branch. By Lemma~\ref{mainlifting} we know that we can lift the elementary embedding to $j:M[G][H]\to N[G][H][G'][\tilde H]$ where $G'$ is  $j(\M)/G$-generic over $V[G][H][K]$ and $\tilde H$ is $j(\P_\alpha)$-generic over the model $V[G][H][K][G']$. Note that there is a $j(\M)$-generic $\tilde G$ over $V$ such that $V[\tilde G]=V[G][H][G']$ by Lemma \ref{absorption}. We need to lift this embedding to have domain $M[G \ast H \ast K]$. The poset $\Add(\kappa,j(\lambda))$ factors as a product $\Add(\kappa,\lambda) \times \Add (\kappa,A)$ where $A =[\lambda,j(\lambda))$. Let $K'$ be $\Add (\kappa,A)$-generic over $N[\tilde G ][\tilde H][K]$. Because $j[K]= K \subset K \times K'$, since $j$ is the identity below $\lambda$, we can lift this to an embedding $j:M[G][H][K] \to N[\tilde G][\tilde H][\tilde K]$ where $\tilde K := K \times K'$. By the classical argument (as in Lemma~\ref{wc-tp}), $T$ has a cofinal branch $b$ in $N[\tilde G][\tilde H][\tilde K]$.

It remains to show that $b \in N[G][H][K]$. First we show that it is already in $N[\tilde G][\tilde K]$. Because $j(\P_\alpha)$ is $\kappa^+$-closed over $N[\tilde G]$ and $\Add(\kappa,j(\lambda))$ has conditions of size $<\kappa$ and a definition that is not changed by $j(\P_\alpha)$, Observation~\ref{product-trick} tells us that $j(\P_\alpha) \ast \Add(\kappa,j(\lambda))$ can be interpreted as a product $j(\P_\alpha) \times \Add(\kappa,j(\lambda))$ over $N[\tilde G]$. Note that in $N[\tilde G]$, $\lambda$ is an ordinal of cofinality $\kappa^+$, and therefore the tree $T$ now has height of cofinality $\kappa^+$ and levels of size at most $\kappa^+$ (in $N[\tilde G]$ but also in $N[\tilde G][\tilde K]$ since $\Add(\kappa,j(\lambda))$ preserves cardinals over $N[\tilde G]$). Since $N[\tilde G] \models \text{``}2^\kappa >\kappa^+$'' and $\Add(\kappa,j(\lambda)) \times j(\P_\alpha)$ is a product of a $\kappa^+$-cc poset with a $\kappa^+$-closed poset over $N[\tilde G]$, Fact~\ref{F:ccc_Closed} implies that $b \in N[\tilde G][ \tilde K]$.

Rewriting $\tilde G$ and $\tilde K$, we now have $b \in N[G][H][G'][K][ K']$. Since $\Add(\kappa,A)$ is square-$\kappa^+$-cc, Fact~\ref{F:Square_cc} implies that $b \in N[G][H][G'][K]$. Because $j(\M)/G$ is $\kappa$-closed, we apply Observation~\ref{product-trick} to write $N[G][H][G' ][K]$ as $N[G][H][K][G']$. Working in $N[G]$, $j(\M)/G$ is a projection of a product of $\Add(\kappa,j(\lambda))$ (a $\kappa^+$-cc forcing) and a $\kappa^+$-closed forcing $\U$. This means that in $N[G][H]$, $\Add(\kappa,A) \times j(\M)/G$ is the projection of $\Add(\kappa,j(\lambda)) \times \U$ where $\U$ is still $\kappa^+$-closed. Hence, by a combination of Fact~\ref{F:ccc_Closed} (applied to $\U$) and Fact~\ref{F:Square_cc} (applied to the copy of $\Add(\kappa,j(\lambda))$ projecting onto $j(\M)/G$) it follows that $b \in N[G][H][K]$. 

\emph{Case 2:} $W$ is the model of $\CSR(\kappa^{++}) \wedge \TP(\kappa^{++}) \wedge \neg \AP(\kappa^{++})$ from Theorem \ref{tp-csr-notap}.

Other than the fact that we are using the lift from Lemma~\ref{mainlifting-notap}, the argument is virtually identical to that of Case 1 until the last paragraph. Then we find that the branch $b \in N[G][H][I][G'][K][K']$, and we want to show that $b \in N[G][H][K]$. Showing that $b \in N[G][H \times I][G'][K]$ follows from an application of Fact~\ref{F:Square_cc}. Then we work in $N[G][H \times I]$ to use the fact that $\N_{\lambda^*}$ is $\kappa$-closed to rewrite $ N[G][H \times I][G'][K] = N[G][H \times I][K][G']$. Since $\N_{\lambda^*}$ is a projection of a product of a $\kappa^+$-closed forcing with a square-$\kappa^+$-cc forcing, we see that $\Add(\kappa,\lambda) \times \N_{\lambda^*}$ is a projection of a $\kappa^+$-closed forcing with a square-$\kappa^+$-cc forcing, so we find that $b \in N[G][H \times I][K]$. This model be written as $N[G][H][K][I]$, so we apply Fact~\ref{F:Square_cc} a final time conclude that $b \in N[G][H][K]$.\end{proof}

\subsection{The Mahlo Results with Large $2^\kappa$}

In this section we prove Theorem~\ref{mahlo-large-cont}, i.e. that the conclusions we derived for the models from Theorems~\ref{wtp-sr-ap}, \ref{gk}, and \ref{sr-notsh} still hold after forcing over these models with $\Add(\kappa,\mu)$ where $\mu$ is an arbitrarily large cardinal of cofinality greater than $\kappa$.

We use another result of the third author (see \cite{HS:u}):

\begin{fact}\label{th:sr}
Suppose $\nu$ is a regular cardinal, $\SR(\nu^+)$ holds and $\Q$ is $\nu$-cc. Then $\Q$ preserves $\SR(\nu^+)$ . 
\end{fact}

Therefore, we find that $\SR(\lambda)$ still holds in these models after forcing with $\Add(\kappa,\mu)$. As in the previous section, we also know that $\neg \SH(\kappa^{++})$ holds after forcing with $\Add(\kappa,\mu)$ over the model from Theorem~\ref{sr-notsh}. 

For the proof of the failure of approachability, we again use the Gitik and Krueger preservation theorem for $\kappa$-centered forcings.

\begin{lemma} If $W$ is either the model from Theorem~\ref{gk} or the model from Theorem~\ref{sr-notsh}, then $W[\Add(\kappa,\mu)] \models \neg \AP(\kappa^{++})$.\end{lemma}

\begin{proof}
This is an immediate consequence of Fact \ref{th:AP}, Lemma \ref{Cohen:centered} and Proposition \ref{AP-trick}.
\end{proof}

The rest of the section will work towards showing that we can maintain $\wTP(\kappa^{++})$ after forcing with $\Add(\kappa,\mu)$ over the model from Theorem \ref{wtp-sr-ap}.

\begin{lemma}\label{fancylemma}  Suppose that $W \subset W'$ are transitive models of (enough of) $\ZFC$ such that $\P,\Q \in W$, $\nu$ is regular in $W'$, $W' \models \text{``}\Q$ is $\nu$-cc'', $W' \models W^{<\nu} \subset W$, and $W' \models \text{``}\P$ has a $\nu$-closed dense subset''. Then $\P$ is $\nu$-distributive over any extension of $W$ by $\Q$.\end{lemma} 

\begin{proof} Because $\P$ is forcing-equivalent to its $\nu$-closed dense set in $W'$, it follows from Easton's Lemma that $W'[\Q] \models \text{``}\P$ is $\nu$-distributive''. Now suppose that $\dot f \in W$ is a $\Q \times \P$-name for a function $\dot f: \tau \to \Ord$ where $\tau < \nu$. If $f$ is the evaluation of $\dot f$ in $W'[\Q \times \P]$, it follows from distributivity that $f \in W'[\Q]$. It is a standard fact that $W' \models \text{``}W^{<\nu} \subset W$ and $W' \models \text{``}\Q$ is $\nu$-cc'' together imply that $W'[\Q] \models \text{``}W[\Q]^{<\nu} \subset W[\Q]$'' \cite{CUMhandbook}. It then follows that $f \in W[\Q]$.\end{proof}

\begin{lemma}\label{wtp-large-cont} If $W$ is the model from Theorem~\ref{wtp-sr-ap}, then $W[\Add(\kappa,\mu)] \models \wTP(\lambda)$.\end{lemma}

\begin{proof} Assume for contradiction that some generic for $\M*\P_{\lambda^+} \ast \Add(\kappa,\mu)$ adds a special $\lambda$-Aronszajn tree. By using a version of Proposition~\ref{a-tree-trick} for special Aronszajn trees, there is some $\alpha<\lambda^{+}$ such that the tree is added by $\M*\P_\alpha*\Add(\kappa,\lambda)$. In other words, there is an $\M*\P_\alpha*\Add(\kappa,\lambda)$-name $\dot{T}$ and some $(m,p,c) \in \M* \P_\alpha*\Add(\kappa,\lambda)$ forcing that $\dot{T}$ is a $\lambda$-Aronszajn tree with a specializing function $\dot{f}$. Let $\Theta$ be sufficiently large with $N:=H_\Theta$ and let $M$ be a rich elementary submodel of $N$ with $\dot{T}$, $\dot{f}$, $(m,p,c)$, and $\alpha$ in $M$.

We first find an  $\M \ast \P_\alpha \ast \Add(\kappa,\lambda)$-generic $G*H*K$ containing $(m,p,c)$ such that the model $N[G][H][K]$ will define a lift $j:\bar M[\bar G][\bar H][\bar K] \to N[G][H][K]$, but we will need to exercise some care in doing so. First, let $\bar G \ast \bar H$ be an $\bar{\M} \ast \bar{\P}_\alpha$-generic over $N$ containing $(\bar m,\bar p)$, and use Lemma~\ref{mainlifting2}(\ref{mainlifting2-lift}) to choose an $\M \ast \P_\alpha$-generic $G*H$ such that $N[G][H]$ defines a lift $\bar M[\bar G][\bar H] \to N[G][H]=N[\bar G][\bar H][G'][H]$, where $G'$ is $\M/\bar G$-generic over $N[\bar G][\bar H]$. In particular, $G*H$ contains $(m,p)$. Let $K$ be an $\Add(\kappa,\lambda)$-generic over $N[G][H]$ such that $c \in K$. Since $j(\Add(\kappa,\bar \lambda)) = \Add(\kappa,\lambda) = \Add(\kappa,\bar \lambda) \times \Add(\kappa,A)$ where $A =[\bar \lambda,\lambda)$, $K=\bar K\times K'$. We can work inside $N[G][H][K]$ and use Fact~\ref{weakliftinglemma} to define a lift $j:\bar M[\bar G][\bar H][\bar K] \to N[G][H][K]$.

Let $T$ and $f$ be the respective interpretations of $\dot{T}$ and $\dot{f}$ in $N[G][H][K]$. As before, let $\bar T,\bar f \in \bar{M}[\bar G][\bar H][\bar K]$ be such that $j(\bar T) = T$ and $j(\bar f) = f$.  Using the classical argument (as in Lemma~\ref{Mahlo-wtp}), we know that $N[G][H][K]$ contains a cofinal branch $b$ of $\bar T$. It will be sufficient to show that $b \in N[\bar G][\bar H][\bar K]$ because $\bar{\M} \ast \bar{\P}_\alpha \ast \Add(\kappa,\bar \lambda)$ preserves $\bar \lambda$, and so we obtain a contradiction as in Lemma~\ref{Mahlo-wtp}.

We will argue now that $b \in N[G][K]$. By Claim~\ref{mainlifting2-denseclosedset}, $N[\bar G] \models \text{``}\bar{\P}_\alpha$ has a $\bar \lambda$-closed dense subset''. Moreover, $N[\bar G] \models$``$\bar M[\bar G]^{<\bar \lambda} \subset \bar M[\bar G]$'' because $N \models $``${\bar M}^{<\bar \lambda} \subset \bar M$'' and because $N \models \text{``}\bar \M$ is $\bar \lambda$-cc''. Furthermore, $N[\bar G] \models \text{``}\Add(\kappa,\bar \lambda)$ is $\bar \lambda$-cc''. Therefore we can apply Lemma~\ref{fancylemma} to find that $\bar{M}[\bar G] \models \text{``}\bar \P_\alpha$ is $\bar \lambda$-distributive over any extension by $\Add(\kappa,\bar \lambda)$''. It follows from elementarity of the embedding $j:\bar{M}[\bar G] \to N[G]$ that $N[G] \models \text{``}\P_\alpha$ is $\lambda$-distributive over any extension by $\Add(\kappa,\lambda)$''. We conclude that $b \in N[G][K]$ since $|b|<\lambda$.

Next we argue that $b \in N[G][\bar K]$. Recall that $N[G][K]=N[ G][\bar K][K']$, where $K'$ is $\Add(\kappa,A)$-generic. Note that $\Add(\kappa,A)$ is square-$\kappa^+$-cc in $N[ G][\bar K]$ since $\kappa^{<\kappa}=\kappa$. Hence Fact~\ref{F:Square_cc} implies that $b \in N[G][\bar K]$.

Now write $N[G][\bar K] = N[\bar G][\bar H][G'][\bar K]$. Because $\M/\bar G$ is $\kappa$-distributive, conditions in $\Add(\kappa,\bar \lambda)$ have size $<\kappa$, and $\M/\bar G$ does not change the definition of $\Add(\kappa,\bar \lambda)$, we can write $N[\bar G][\bar H][G'][\bar K] = N[\bar G][\bar H][\bar K][G']$ by Remark~\ref{product-trick}. Working in $N[\bar G]$, $\M/\bar G$ is a projection of a product of $\Add(\kappa,\lambda)$ and a $\kappa^+$-closed forcing $\U$, hence $\Add(\kappa,\bar \lambda) \times \M/G$ is a projection of $\Add(\kappa, \lambda) \times \U$ in $N[\bar G][\bar H]$ where $\U$ is still $\kappa^+$-closed. Then Fact~\ref{F:ccc_Closed} can be applied with $N[\bar G][\bar H]$ as the ground model to find that $b$ is in an extension of $N[\bar G][\bar H][\bar K]$ by a square-$\kappa^+$-cc forcing. Finally, Fact~\ref{F:Square_cc} implies that $b \in N[\bar G][\bar H][\bar K]$.\end{proof}

\subsection*{An Open Question} 

The reader may be concerned about the fact that despite the effort it took to show that $\SH(\kappa^{++})$ fails in the model of Theorem~\ref{sr-notsh}, we have not settled the question of $\SH(\kappa^{++})$ in the other Mahlo models.

\begin{question*}\label{sh-question} Does $\SH(\kappa^{++})$ hold in the models of Theorems \ref{wtp-sr-ap} and \ref{gk}? This is assuming that the $\lambda$ we begin with is \emph{not} weakly compact.\end{question*}

If $\lambda$ is regular and not weakly compact in $L$ and $\GCH$ holds, then there are $\lambda$-Suslin trees. More precisely, work of Jensen and Todor{\v c}evi{\'c} showed that if $\lambda$ is regular and not weakly compact in $L$, then the principle $\square(\lambda)$ holds in $V$ \cite{T:partitioning}; and then Rinot showed that $\square(\lambda)$ and $\GCH$ together imply the existence of $\lambda$-Suslin trees \cite{RINOT2017510}. From this reasoning it also follows that $L[\M] \models \neg \SH(\kappa^{++})$ if $\M$ is one of our Mitchell forcings. Hence, a positive answer to this question would necessarily make use of the interaction between our Mitchell forcing and the club-adding iteration that follows it. It would also confirm that  $\GCH$ is strictly necessary in the result of Rinot.

\subsection*{Acknowledgements:} We would like to thank Monroe Eskew for recognizing a problem with an earlier version of Lemma~\ref{fancylemma}---even finding a subtle counterexample to the original statement involving the Laver collapse.

\end{document}